\documentclass[reqno,12pt]{amsart}
\usepackage{amsmath,amssymb,latexsym,soul,cite,mathrsfs}
\usepackage{blindtext}
\usepackage{enumitem}
\usepackage[utf8]{inputenc}
\usepackage{amsthm,amsfonts}
\usepackage{comment}
\usepackage[mathcal]{eucal}
\usepackage{latexsym}
\usepackage[utf8]{inputenc}
\usepackage{bm}
\usepackage[all]{xy}
\usepackage{xcolor}
\usepackage{graphicx}
\usepackage{hyperref}

\newtheorem{proposition}{Proposition}

\newtheorem{theorem}{Theorem}

\newtheorem*{definition}{Definition}

\newtheorem{theo}{Theorem}[section]

\newtheorem{remark}{Remark}[section]

\newtheorem{lemma}{Lemma}[section]

\newtheorem{prop}{Proposition}[section]

\newtheorem{cor}{Corollary}[section]

\numberwithin{equation}{section}

\title[TRUDINGER-MOSER EMBEDDING ON WEIGHTED SOBOLEV SPACES]{TRUDINGER-MOSER EMBEDDINGS ON WEIGHTED SOBOLEV SPACES on UNBOUNDED DOMAINS}

\author[J.M.\ do \'O]{Jo\~ao Marcos do \'O}
\author[G. Lu]{Guozhen Lu}
\author[R.~Ponciano]{Raon\'{\i} Ponciano}

\address[Jo\~{a}o Marcos do \'O]{Dep. Mathematics,
	Federal University of Para\'{\i}ba
	\newline\indent
	58051-900, Jo\~ao Pessoa-PB, Brazil}
\email{\href{mailto:jmbo@pq.cnpq.br}{jmbo@pq.cnpq.br}}

\address[Guozhen Lu]{Dep. Mathematics, University of Connecticut
	\newline\indent
	06269, Storrs-CT, United States of America}
\email{\href{mailto:guozhen.lu@uconn.edu}{guozhen.lu@uconn.edu}}

\address[Raon\'{\i}~Ponciano]{Dep. Mathematics,
	Federal University of Para\'{\i}ba
	\newline\indent
	58051-900, Jo\~ao Pessoa-PB, Brazil}
\email{\href{mailto:raoni.cabral.ponciano@academico.ufpb.br}{raoni.cabral.ponciano@academico.ufpb.br}}

\thanks{This work was partially supported by Conselho Nacional de Desenvolvimento Cient\'ifico e Tecnol\'ogico (CNPq) \#312340/2021-4 and \#429285/2016-7, Funda\c c\~ao de Apoio \`a Pesquisa do Estado da Para\'iba (FAPESQ) \#2020/07566-3, Coordena\c c\~ao de Aperfei\c coamento de Pessoal de N\'ivel Superior (CAPES) \#88887.633572/2021-00 and Simons collaboration grants 519099 and 957892 and Simons Fellowship from Simons foundation}

\subjclass[2000]{35J20, 35J25, 35J50}

\keywords{Sobolev spaces; Trudinger-Moser type inequality; Differential Equations; Fractional Dimensions; Maximizers}

\begin{document}

\begin{abstract}
We establish embeddings on a class of Sobolev spaces with potential weights on unbounded domains. Our results provide embeddings into weighted Lebesgue spaces $L^q_\theta$ with radial power weights and establish the existence and non-existence of the maximizers for their Trudinger-Moser type inequalities. We also sharpen the maximal integrability by ``removing" necessary terms from the exponential series while maintaining the continuity of the embedding.
\end{abstract}

\maketitle

\section{Introduction}
Recently, the weighted Sobolev space $X^{k,p}_R$ has been extensively studied due to its applicability in radial elliptic problems for operators in great generality. For more details, see \cite{MR1982932,MR1422009,MR4112674,MR3209335,MR3575914,JMBO,JMBO2,Paper1,MR1069756,MR3670473,MR3957979,MR1929156,MR2166492} and references therein. In \cite{MR1422009}, the authors study a  Br\'{e}zis-Nirenberg type problem for a class of quasilinear elliptic operator of the form  $Lu=-r^{-\gamma}(r^\alpha|u'|^\beta u')'$ when considered as acting on radial functions defined on the ball centered at the origin with radius $R$. According to the choice of parameters, the following operators are included in the class:
\begin{enumerate}
    \item[(i)] $L$ is the Laplacian for $\alpha=\gamma=N-1$ and $\beta=0$;
    \item[(ii)] $L$ is the $p$-Laplacian for $\alpha=\gamma=N-1$ and $\beta=p-2$;
    \item[(ii)] $L$ is the $k$-Hessian for $\alpha=N-k$, $\gamma=N-1$ and $\beta=k-1$.
\end{enumerate}
The suitable space to work on problems like
\begin{equation*}
\left\{\begin{array}{ll}
Lu=f(r,u)\mbox{ in }(0,R)&  \\
u'(0)=u(R)=0,&\\
u>0\mbox{ in }(0,R).
\end{array}\right.
\end{equation*}
is the weighted Sobolev space $X^{1,p}_{0,R}$ to be defined below.

Let $0<R\leq\infty$ and $AC(0,R)$ be the space of all absolutely continuous functions on interval $(0,R)$. It is well known that $u\in AC(0,R)$ if and only if $u$ has a derivative $u'$ almost everywhere, which is Lebesgue integrable and $u(r)=u(a)+\int_a^r u'(s) \mathrm ds$. A function $u$ is said to be locally absolutely continuous on $(0,R)$ if for every $r\in (0,R)$ there exists a neighborhood $V_r$ of $r$ such that $u$ is absolutely continuous on $V_r$. Let $AC_{\mathrm{loc}}(0,R)$ be the space of all locally absolutely continuous functions on interval $(0,R)$.

For each non-negative integer $\ell$, let $AC_{\mathrm{loc}}^\ell(0,R)$ be the set of all functions $u\colon(0,R)\to \mathbb R$ such that $u^{(\ell)}\in AC_{\mathrm{loc}}(0,R)$, where $u^{(\ell)}=\mathrm d^\ell u/\mathrm dr^\ell$. For $p\geq1$ and $\alpha$ real numbers, we denote by $L^p_\alpha=L^p_\alpha(0,R)$ the weighted Lebesgue
 space  of measurable functions $u\colon(0,R)\to\mathbb R$ such that
\begin{equation*}
\|u\|_{L^p_\alpha}=\left(\int_0^R|u|^pr^\alpha\mathrm dr\right)^{1/p}<\infty,
\end{equation*}
which is a Banach space under the standard norm $\|u\|_{L_\alpha^p}$.

The borderline case of the Sobolev embeddings is well-known in the literature as Trudinger-Moser type inequalities which were discovered by J. Peetre \cite{MR0221282}, S. Pohozaev \cite{MR0192184}, N. Trudinger \cite{MR0216286} and V. Yudovich \cite{MR0140822}. Essentially they proved that for some small positive $\alpha$ the Sobolev space $W^{1,N}_0(\Omega)$ is continuously embedding into the Orlicz space $L_{\psi_N}(\Omega)$, where $\Omega$ is a smooth bounded domain of $\mathbb R^N$ and $\psi_N(t)=\exp(\alpha |t|^{\frac{N}{N-1}})$. J. Moser \cite{MR0301504} sharpened this result. Using a symmetrization argument he was able to find the optimal $\alpha=\alpha_N=N\omega_{N-1}^{\frac{1}{N-1}}$, where $\omega_{N-1}$ is the $(N-1)$-dimensional surface of the unit sphere. Trudinger-Moser inequality in unbounded domains of $\mathbb{R}^N (N \geq 2)$ was established in \cite{MR1163431,MR2109256} for the case $N=2$ and for $N \geq 2$ see \cite{MR1377667,MR1704875,MR2400264,MR1646323}. In other settings such as the Heisenberg group where the P\'olya-Szeg\"{o} inequality does not work, Trudinger-Moser inequalities on both bounded and unbounded domains
can be proved by using a symmetrization-free argument. See \cite{CohnLu, LamLu3, LamLuTang-NA, LiLuZhu-CVPDE}. Moreover, such a symmetrization-free argument also offers  alternative proofs for both critical and subcritical  Trudinger-Moser inequalities in the Euclidean spaces. (See also \cite{LLZ-ANS}).

The existence of solutions to nonlinear elliptic equations involving subcritical and critical growth of the Trudinger-Moser type has been studied extensively in recent years, motivated by its applicability in many fields of modern mathematics; see \cite{MR2772124} for a survey on this subject. We also mention the articles  \cite{CLZ-AIM, CLZ-CVPDE, MR1704875,Paper1,MR3670473,MR1163431,MR1377667,MR1865413,MR2863858,MR1392090,MR2488689, LLZ-AIM, LuZhu-JDE} and therein.

For the first order  weighted Sobolev space, the authors in \cite{MR3209335} study Trudinger-Moser type inequalities for weighted Sobolev spaces on finite intervals $(0,R)$, these estimates can be interpreted as Trudinger-Moser type inequalities in fractional dimensions. See also \cite{MR3299177} where the same authors establish a sharp form of the Trudinger-Moser type inequality for weighted Sobolev spaces.

For any positive integer $k$ and $(\alpha_0,\ldots,\alpha_k)\in \mathbb R^{k+1}$, with $\alpha_j>-1$ for $j=0,1,\ldots,k$,   the authors considered in \cite{MR4112674} the weighted Sobolev spaces for higher-order derivatives $X^{k,p}_{0,R}=X^{k,p}_{0,R}(\alpha_0,\ldots,\alpha_k)$ given by all functions $u\in AC^{k-1}_{\mathrm{loc}}(0,R)$ such that $u^{(j)}\in L^p_{\alpha_j}$ for all $j=0,\ldots,k$ and satisfying the boundary condition
\begin{equation*}
\lim_{r\to R}u^{(j)}(r)=0,\quad j=0,\ldots,k-1.
\end{equation*}

Recently D. de Figueiredo et al \cite{MR2838041} obtained a new class on Sobolev embeddings in the spaces of radial functions and applied these embeddings to prove the existence of radial solutions for a biharmonic equation of the H\'{e}non type under both Dirichlet and Navier boundary conditions. Motivated by this work the authors in \cite{Paper1} considered the weighted Sobolev spaces for higher-order derivatives without boundary conditions:
\begin{equation*}
X_R^{k,p}\!=\!X_R^{k,p}(\alpha_0,\ldots,\alpha_k)\!=\!\{u\colon(0,R)\to\mathbb R : u^{(j)}\in L^p_{\alpha_j},\ j=0,1,\ldots,k\}.
\end{equation*}
We emphasize that in the above definition of $X^{k,p}_R$ it is considered the derivative in the weak sense. Using Proposition \ref{prop31} below, one can obtain $u\in AC_{\mathrm{loc}}^{k-1}(0,R)$ for all $u\in X^{k,p}_R$. The spaces $X_R^{k,p}$ and $X^{k,p}_{0,R}$ are complete under the norm
\begin{equation*}
\|u\|_{X_R^{k,p}}=\left(\sum_{j=0}^k\|u^{(j)}\|^p_{L^p_{\alpha_j}}\right)^{1/p}.
\end{equation*}

Next, we state that the definitions for $X^{k,p}_R$ using weak derivatives or absolutely continuous functions are equivalent (see \cite[Proposition 2.2]{Paper1}).

\begin{proposition}\label{prop31}
Suppose $1\leq p<\infty$ and $0<R\leq\infty$. Then $u\in X_R^{k,p}(\alpha_0,\alpha_1,\ldots,\alpha_k)$ holds if and only if $u=U$ a.e. in $(0,R)$ for some $U\in AC_{\mathrm{loc}}^{k-1}((0,R])$ such that $U^{(k)}(t)$ (in the classical sense) exists a.e. $t\in (0,R)$ and $U^{(k)}$ is a measurable function and
\begin{equation*}
\int_0^R\left|U^{(j)}(t)\right|^pt^{\alpha_j}\mathrm dt<\infty\mbox{ for }j=0,1,\ldots,k.
\end{equation*}
\end{proposition}

For $0<R<\infty$, J. M. do \'O, A. C. Macedo, and J. F. de Oliveira \cite{MR4112674} established the Sobolev embedding for $X^{k,p}_{0,R}$ into $L^q_\theta$ (see \cite[Theorem 1.1]{MR4112674}). In \cite[Theorem 1.1]{Paper1}, the authors of the current paper, considering the weighted Sobolev space $X^{k,p}_R$ without boundary condition, proved that

\begin{theorem}\label{theo32}
Let $p\geq1$, $0<R<\infty$ and $\theta\geq\alpha_k-kp$.
\begin{flushleft}
$\mathrm{(a)}$ Sobolev case: If $ \alpha_k-kp+1>0$, then the continuous embedding holds
\begin{equation*}X^{k,p}_R(\alpha_0,\ldots,\alpha_k)\hookrightarrow L^q_\theta(0,R)\quad \text{for all}\quad 1\leq q\leq p^*:=\dfrac{(\theta+1)p}{\alpha_k-kp+1}. \end{equation*}
Moreover, the embedding is compact if $q<p^*.$ \\
$\mathrm{(b)}$ Sobolev Limit case: If $\alpha_k-kp+1=0$, then the compact embedding holds
\begin{equation*}
    X^{k,p}_R(\alpha_0,\ldots,\alpha_k)\hookrightarrow L^q_\theta(0,R)\quad \text{for all} \quad q\in[1,\infty).
\end{equation*}
\end{flushleft}
\end{theorem}

In this paper, we focus on the weighted Sobolev space $X^{k,p}_R$ with $R=\infty$ which generalizes the space of radial functions in $W^{k,p}(\mathbb R^N)$ because all weights are not necessarily equal to $r^{N-1}$. We establish, similar to Theorem \ref{theo32}, the embedding $X^{k,p}_\infty\hookrightarrow L^q_\theta$ in our first main result which is the following
\begin{theo}\label{theoimersaoinfinito}
Let $X^{k,p}_\infty(\alpha_0,\ldots,\alpha_k)$ be the weighted Sobolev space with $p\geq1$ and $\alpha_{j}\geq\alpha_k-(k-j)p$ for all $j=0,\ldots,k-1$. Assume $\alpha_k-kp\leq\theta\leq\alpha_0$.
\begin{flushleft}
    $\mathrm{(a)}$ Sobolev embedding: If $\alpha_k-kp+1>0$, then the following continuous embedding holds:
    \begin{equation*}
    X^{k,p}_\infty\hookrightarrow L^q_\theta,\quad p\leq q\leq p^*,
    \end{equation*}
    where
    \begin{equation*}
    p^*:=\dfrac{(\theta+1)p}{\alpha_k-kp+1}.
    \end{equation*}
    Moreover, it is a compact embedding if one of the following two conditions is fulfilled:
    \begin{flushleft}
        $\mathrm{(i)}$ $\theta=\alpha_0$ and $p<q<p^*$;\\
        $\mathrm{(ii)}$ $\theta<\alpha_0$ and $p\leq q<p^*$.
    \end{flushleft}
    $\mathrm{(b)}$ Limiting Sobolev embedding: If $\alpha_k-kp+1=0$, then the following continuous embedding holds:
    \begin{equation*}
    X^{k,p}_\infty\hookrightarrow L^q_\theta,\quad p\leq q<\infty.
    \end{equation*}
    Moreover, it is compact embedding if one of the following two conditions is fulfilled:
    \begin{flushleft}
        $\mathrm{(i)}$ $\theta=\alpha_0$, $q>p>1$ and $\alpha_0>-1$;\\
        $\mathrm{(ii)}$ $\theta<\alpha_0$ and $q\geq p\geq1$.
    \end{flushleft}
\end{flushleft}
\end{theo}
\begin{remark}
Theorem \ref{theoimersaoinfinito} generalizes \cite[Lemma 2.4]{MR3670473} because their result only works for the first order derivative. Moreover, we do not impose the restrictive condition: $\theta=\alpha_0\geq0$ and $p\geq2$ assumed in \cite{MR3670473}.
\end{remark}

For our next main results, we study the finiteness of the following supremum
\begin{equation}\label{eqdpmu1}
d_{p,\mu,\theta}:=\sup_{\underset{\|u\|_{X^{1,p}_\infty}=1}{u\in X^{1,p}_\infty}}\int_0^\infty\exp_p(\mu|u|^{\frac{p}{p-1}})r^{\theta}\mathrm dr
\end{equation}
and the attainability of the following supremum
\begin{equation}\label{eqdpmu}
d_{p,\mu}:=\sup_{\underset{\|u\|_{X^{1,p}_\infty}=1}{u\in X^{1,p}_\infty}}\int_0^\infty\exp_p(\mu|u|^{\frac{p}{p-1}})r^{\alpha_0}\mathrm dr
\end{equation}
where
\begin{equation*}
    \exp_p(t):=\sum_{j=0}^\infty\dfrac{t^{p-1+j}}{\Gamma(p+j)}.
\end{equation*}

\begin{remark}\label{remark1}
Let us discuss the choice of the function $\exp_p$ briefly. When $p\in\mathbb N$ it is necessary to remove some terms in the Taylor expansion of the function $\exp$ to guarantee that
\begin{equation*}
\int_0^\infty\exp(\mu|u|^{\frac{p}{p-1}})r^{\alpha_0}\mathrm dr
\end{equation*}
is finite for each $u\in X_\infty^{1,p}$. However, when $p\notin\mathbb N$ (which corresponds to the fractional dimension case) many authors \cite{MR3670473,MR3209335,JMBO,JMBO2,MR4097244, LLZ-ANS, LLZ-RMI} considered the function
\begin{equation*}
\phi_p(t):=\mathrm{e}^{t}-\sum_{j=0}^{\lceil p-2\rceil}\dfrac{t^j}{j!}=\sum_{j=\lceil p-1\rceil}^{\infty}\dfrac{t^j}{j!},
\end{equation*}
where $\lceil x\rceil=\min\{n\in\mathbb Z\colon n\geq x\}$ denotes the ceiling function applied on $x\in\mathbb R$, and studied the supremum
\begin{equation*}
\sup_{\underset{\|u\|_{X^{1,p}_\infty}=1}{u\in X^{1,p}_\infty}}\int_0^\infty\phi_p(\mu|u|^{\frac{p}{p-1}})r^{\alpha_0}\mathrm dr.
\end{equation*}
Theorems that involve $\phi_p$ are weaker than those that involve $\exp_p$ because $\phi_p(t)=\exp_{\lceil p\rceil}(t)<\exp_p(t)$ for $t>0$ and $p\notin\mathbb N$ (see Lemma \ref{lemmaphiN}). Furthermore, comparing the differences between
\begin{equation*}
\phi_p(|u|^{\frac{p}{p-1}})=\dfrac{|u|^{\frac{p\lceil p-1\rceil}{p-1}}}{\lceil p-1\rceil!}+\dfrac{|u|^{\frac{p\lceil p\rceil}{p-1}}}{\lceil p\rceil!}+\cdots
\end{equation*}
and
\begin{equation*}
\exp_p(|u|^{\frac{p}{p-1}})=\dfrac{|u|^{p}}{\Gamma(p)}+\dfrac{|u|^{\frac{p^2}{p-1}}}{\Gamma(p+1)}+\cdots.
\end{equation*}
and from the compact embedding $X^{1,p}_\infty(\alpha_0,p-1)\hookrightarrow L^q_\theta$ for $p<q<\infty$ (see Theorem \ref{theoimersaoinfinito}) and $p\lceil p-1\rceil/(p-1)>p$, we have that $\phi_p$ fits in the compact case and $\exp_p$ fits just in the continous case. Therefore, $\exp_p$ is the biggest power series function (up to multiplication by $|u|^\alpha$) which only fits in the continuous case.
\end{remark}

Analogously, the argument from Remark \ref{remark1} can be used for the embedding $W^{\gamma,\frac{N}{\gamma}}(\mathbb R^N)\hookrightarrow L^q(\mathbb R^N)$ where $0<\gamma<N$ and $p\leq q<\infty$. For instance, \cite[Theorem 1.7]{MR3053467} in the work of N. Lam and G. Lu can be sharpened into the following theorem and its proof follows by the same argument as in \cite{MR3053467} with $\exp_p(t)\leq \mathrm{e}^t$ which is a consequence of Lemma \ref{lemmaphiN}.
\begin{theo}
Let $0<\gamma<N$ be an arbitrary real positive number, $p=\frac{N}{\gamma}$, and $\tau>0$. There holds
\begin{equation*}
\sup_{\underset{\|(\tau I-\Delta)^{\frac{\gamma}2}u\|_p\leq1}{u\in W^{\gamma,p}(\mathbb R^N)}}\int_{\mathbb R^N}\exp_p\left(\beta_0(N,\gamma)|u|^{\frac{p}{p-1}}\right)\mathrm dx<\infty
\end{equation*}
where $\omega_{N-1}$ is the area of the surface of the unit ball $N$-ball and
\begin{equation*}
\beta_0(N,\gamma)=\dfrac{N}{\omega_{N-1}}\left[\dfrac{\pi^{\frac{N}2}2^\gamma\Gamma\left(\frac{\gamma}{2}\right)}{\Gamma\left(\frac{N-\gamma}{2}\right)}\right]^{\frac{p}{p-1}}.
\end{equation*}
Furthermore, this inequality is sharp, i.e., if $\beta_0(N,\gamma)$ is replaced by any $\beta>\beta_0(N,\gamma)$, then the supremum is infinite.
\end{theo}

Recently, the authors in \cite[Theorem 1.2]{Paper1}, assuming the case $R<\infty$, have investigated the finiteness and attainability of
\begin{equation*}
\ell_{\mu}:= \sup_{\|u\|_{X_R^{k,p}}\leq1}\int_0^R\mathrm{e}^{\mu|u|^{\frac{p}{p-1}}}r^\theta \mathrm dr
\end{equation*}
for the subcritical case. More specifically, they concluded the following theorem.

\begin{theorem}\label{theo01}
Let $X^{k,p}_R(\alpha_0,\ldots,\alpha_k)$ with $R<\infty$, $\alpha_k-kp+1=0$, $p>1$ and $\theta>-1$. Set $\mu_0:=(\theta+1)[(k-1)!]^{p/(p-1)}$.
\begin{flushleft}
$\mathrm{(a)}$ For all $\mu\geq0$ and $u\in X^{k,p}_R,$ we have $\exp(\mu|u|^{p/(p-1)})\in L^1_\theta$.\\
$\mathrm{(b)}$ If $0\leq\mu<\mu_0,$ then $\ell_{\mu}$ is finite. Moreover, $\ell_{\mu}$  is attained by a nonnegative function $u_0 \in X^{k,p}_R$ with $\|u_0\|_{X^{k,p}_R}=1$.\\
$\mathrm{(c)}$ If $\mu>\mu_0$ and $\alpha_i-ip+1>0$ for $i=0,\ldots,k-1$, then
$\ell_{\mu}=\infty$.
\end{flushleft}
\end{theorem}

Assuming $k=1$ we are able to establish a similar result to Theorem \ref{theo01} for the case $R=\infty$. Note that \cite[Theorem 1.1]{MR4097244} also considered the unbounded domain case but under the restrictive condition $\theta=\alpha_0$ and $p\geq2$. Next, we establish a Trudinger-Moser embedding on unbounded domains for weighted Sobolev spaces studying the finiteness of \eqref{eqdpmu1}.

\begin{theo}\label{theo0}
Let $\theta>-1$ and $X^{1,p}_\infty(\alpha_0,p-1)$ be the weighted Sobolev space such that $\alpha_0>-1$ and $p>1$.
\begin{flushleft}
$\mathrm{(a)}$ For all $\mu\geq0$ and $u\in X^{1,p}_\infty,$ we have $\exp_p(\mu|u|^{\frac{p}{p-1}})\in L^1_\theta$.\\
$\mathrm{(b)}$ If $0\leq\mu<\theta+1,$ then $d_{p,\mu,\theta}$ is finite.\\
$\mathrm{(c)}$ If $\mu>\theta+1$, then
$d_{p,\mu,\theta}=\infty$.
\end{flushleft}
\end{theo}

In view of item (b) in Theorem \ref{theo01}, we work on the attainability of item (b) in Theorem \ref{theo0}. Based on M. Ishiwata \cite{MR2854113} and assuming $\theta=\alpha_0$, we are able to prove the attainability for \eqref{eqdpmu} in our next main theorem. Firstly, we need to introduce the interpolation constant similar to that considered by M. Weinstein \cite{MR0691044}:
\begin{equation*}
B_{2,\alpha_0}:=\sup_{u\in X^{1,2}_\infty\backslash\{0\}}\frac{\|u\|^4_{L^4_{\alpha_0}}}{\|u'\|_{L^2_1}^2\|u\|_{L^2_{\alpha_0}}^2}.
\end{equation*}
Moreover, we prove in Lemma \ref{lemma33} that $B_{2,\alpha_0}\geq4/(\alpha_0+1)\mathrm e\log 2$ which is a better estimate than $B_{2,\alpha_0}>2/(\alpha_0+1)$ (see \cite[Proposition 7.1]{MR4097244}).

\begin{theo}\label{theo12}
Let $X^{1,p}_\infty(\alpha_0,p-1)$ be the weighted Sobolev space with $p\geq2$ and $\alpha_0>-1$.
\begin{flushleft}
    $\mathrm{(a)}$ If $p>2$ and $0<\mu<\alpha_0+1,$ then $d_{p,\mu}$ is attained;\\
    $\mathrm{(b)}$ If $p=2$ and $2/B_{2,\alpha_0}<\mu<\alpha_0+1,$  then $d_{p,\mu}$ is attained.
\end{flushleft}
\end{theo}
See also related results in \cite{CLZ-AIM}.

In our next result, we establish the non-existence of extremal functions for \eqref{eqdpmu} assuming $1<p\leq 2$ and $\mu>0$ small.

\begin{theo}\label{theo15}
Let $1<p\leq2$ and $\alpha_0>-1$. If $\mu>0$ is sufficiently small, then $d_{p,\mu}$ is not attained. Moreover, the functional
\begin{equation*}
u\mapsto\int_0^\infty\exp_p\left(\mu|u|^{\frac{p}{p-1}}\right)r^{\alpha_0}\mathrm dr,\quad u\in \mathcal S,
\end{equation*}
does not have a critical point on $\mathcal S=\{u\in X^{1,p}_\infty\colon\|u\|_{X^{1,p}_\infty}=1\}$ for $\mu>0$ small.
\end{theo}
\begin{remark}
E. Abreu and L. G. Fernandes Jr \cite{MR4097244} proved Theorem \ref{theo12} with $\alpha_0\geq0$ and $\phi_p$ instead of $\exp_p$ (see Remark \ref{remark1}). Also, they proved Theorem \ref{theo15} assuming $p=2$, $\alpha_0\geq0$, and $\phi_p$ instead of $\exp_p$. We also refer to \cite{LLZ-AIM} for related results in nonradial case.
\end{remark}

The organization of the paper is as follows. In Section \ref{sec2}, we develop a Radial Lemma for unbounded domains (Lemma \ref{lemma53}) to prove Theorem \ref{theoimersaoinfinito} and \ref{theo0}. In Section \ref{sec3}, we work on the attainability of $d_{p,\mu}$. More specifically, we show Theorem \ref{theo12}. We will prove in Section \ref{sec4}, Theorem \ref{theo15} which guarantees that $d_{p,\mu}$ is not attained for $1<p\leq2$ and $\mu>0$ small enough.

\section{Radial Lemma and its consequences}\label{sec2}

Throughout the rest of the paper, we denote $p'=p/(p-1)$ for each $p\in(1,\infty)$.
The following lemma shows that the space of the restriction to the interval $[0,\infty)$ of the smooth functions with compact support in $\mathbb{R}$ denote by  $\Upsilon:=\{u|_{[0,\infty)} : u \in C_0^\infty(\mathbb{R})\}$ is dense in $X^{k,p}_\infty$. This fact will be used to obtain our radial Lemma (Lemma \ref{lemma53}) for $X^{k,p}_\infty$.

\begin{lemma}\label{lemmacompsuppdense}
The space $\Upsilon$ is dense in the weighted Sobolev space $X^{k,p}_\infty$ under the assumptions
$p\geq1$ and $\alpha_{j-1}\geq\alpha_j-p$ for all $j=1,\ldots,k$.
\end{lemma}
\begin{proof}
For each $n\in\mathbb N$ we define the cut-off $\eta_n\colon[0,\infty)\to\mathbb R$,
\begin{equation*}
\eta_n(r)=\left\{\begin{array}{ll}
     1&\mbox{for }0\leq r< n,  \\
     \log\left(\frac{n\mathrm{e}}{r}\right)&\mbox{for }n\leq r<ne,\\
     0&\mbox{for }r\geq ne.
\end{array}\right.
\end{equation*}
For each $u\in X^{k,p}_\infty$ consider $u_n=\eta_nu$. We are going to prove that
\begin{equation}\label{eqgha}
\int_0^\infty|u_n^{(i)}-u^{(i)}|^pr^{\alpha_i}\mathrm dr\overset{n\to\infty}\longrightarrow0,\quad\forall i=0,\ldots,k.
\end{equation}
Since $0\leq\eta_n\leq1$ and $u\in L^p_{\alpha_0}$ we obtain
\begin{equation*}
\int_0^\infty|u_n-u|^pr^{\alpha_0}\mathrm dr\leq\int_n^\infty|u|^pr^{\alpha_0}\mathrm dr\overset{n\to\infty}\longrightarrow0.
\end{equation*}
This concludes \eqref{eqgha} for $i=0$. Let $i=1,\ldots,k$. By $\alpha_{j-1}\geq\alpha_j-p$ for all $j=1,\ldots,k$ we have
\begin{equation}\label{eqgha1}
\alpha_{i-j}\geq \alpha_i-jp,\quad\forall j=1,\ldots,i.
\end{equation}
Using
\begin{equation*}
\int_0^{\infty}|u_n^{(i)}-u^{(i)}|^pr^{\alpha_i}\mathrm dr=\int_n^{n\mathrm{e}}|(\eta_nu)^{(i)}-u^{(i)}|^pr^{\alpha_i}\mathrm dr+\int_{n\mathrm{e}}^\infty|u^{(i)}|^pr^{\alpha_i}\mathrm dr
\end{equation*}
and
\begin{align*}
\int_n^{n\mathrm{e}}|(\eta_nu)^{(i)}-u^{(i)}|^pr^{\alpha_i}\mathrm dr&=\int_n^{n\mathrm{e}}\left|\sum_{j=0}^i\binom{i}{j}u^{(i-j)}\eta_n^{(j)}-u^{(i)}\right|^pr^{\alpha_i}\mathrm dr\\
&\leq(i+1)^p\int_n^{n\mathrm{e}}\left|u^{(i)}(1-\eta_n)\right|^pr^{\alpha_i}\mathrm dr\\
&\quad+(i+1)^p\sum_{j=1}^i\int_n^{n\mathrm{e}}|u^{(i-j)}\eta_n^{(j)}|^pr^{\alpha_i}\mathrm dr\\
&\leq(i+1)^p\int_n^{\infty}\left|u^{(i)}\right|^pr^{\alpha_i}\mathrm dr\\
&\quad+(i+1)^p\sum_{j=1}^i\int_n^{\infty}|u^{(i-j)}|^p|\eta_n^{(j)}|^pr^{\alpha_i}\mathrm dr
\end{align*}
we notice that \eqref{eqgha} follows if
\begin{equation*}
|\eta_n^{(j)}|^pr^{\alpha_i}\leq Cr^{\alpha_{i-j}},\quad\forall 1\leq j\leq i\leq k\mbox{ and }r\in(n,n\mathrm{e}),
\end{equation*}
where $C>0$ does not depending on $n$ and $r\in(n,n\mathrm{e})$. Using induction we have
\begin{equation*}
\eta_n^{(j)}(r)=(-1)^j(j-1)!r^{-j},\quad\forall r\in(n,n\mathrm{e}),j\in\mathbb N.
\end{equation*}
Therefore, by \eqref{eqgha1},
\begin{equation*}
|\eta_n^{(j)}|^pr^{\alpha_i}=(j-1)!r^{\alpha_i-jp}\leq (j-1)!r^{\alpha_{i-j}},\quad \forall r\in(n,n\mathrm{e}).
\end{equation*}
Since each $u_n$ can be approximated uniformly in $C^k([0,\infty))$ by a function in $\Upsilon=\{u|_{[0,\infty)}\colon u\in C^\infty_0(\mathbb R)\}$ we conclude our lemma.
\end{proof}

In the next lemma, we prove a decaying property for the functions in the space  $X^{1,p}_\infty$
and its proof is based on \cite[Lemma II.1]{MR0683027}. We do not assume $p\geq2$ as in \cite[Lemma 2.3]{MR4097244}.

\begin{lemma}\label{lemma53}
Suppose $p\geq1, \; \alpha_0\geq\alpha_1-p$ and $\alpha_1-p+1\geq0$.
Then for any $u\in X^{1,p}_\infty$ it holds,
\begin{equation*}
|u(r)|\leq p^\frac1p\|u\|_{L^p_{\alpha_0}}^{\frac{p-1}p}\|u'\|_{L^p_{\alpha_1}}^{\frac1p}r^{-\frac{\alpha_0(p-1)+\alpha_1}{p^2}},\quad\forall r\in(0,\infty).
\end{equation*}
\end{lemma}
\begin{proof}
Given $u\in X^{1,p}_\infty$, by Lemma \ref{lemmacompsuppdense}, we can assume that there exists $R>0$ such that $u(r)=0$ for $r>R$. Firstly, we claim that $s\mapsto |u(s)|^p$ is an absolutely continuous function on $(r,\infty)$ for a fixed $r\in(0,\infty)$. Proposition \ref{prop31} implies that $u$ is bounded on $[r,R]$ and using that $s\mapsto s^p$ is a Lipschitz function on $[0,\|u\|_{L^\infty(r,R)}]$, we conclude our claim. For now, suppose $p>1$. Using that $[\alpha_0(p-1)+\alpha_1]/p\geq0$ we have
\begin{align*}
|u(r)|^p&=\int_r^\infty\left(|u(s)|^p\right)'\mathrm ds=\int_r^\infty p|u(s)|^{p-2}u(s)u'(s)\mathrm ds\\
&\leq p\int_r^\infty|u(s)|^{p-1}|u'(s)|\left(\dfrac{s}{r}\right)^{\frac{\alpha_0(p-1)+\alpha_1}{p}}\mathrm ds\\
&\leq p\|u\|_{L^p_{\alpha_0}}^{\frac{p-1}{p}}\|u'\|_{L^p_{\alpha_1}}^{\frac{1}{p}}r^{-\frac{\alpha_0(p-1)+\alpha_1}{p}}.
\end{align*}

The case $p=1$ immediately follows from
\begin{equation*}
|u(r)|=\left|\int_r^\infty u'(s)\mathrm ds\right|\leq\int_r^\infty|u'(s)|\left(\dfrac{s}{r}\right)^{\alpha_1}\mathrm ds=\|u'\|_{L^1_{\alpha_1}}r^{-\alpha_1}.
\end{equation*}
\end{proof}

With the Radial Lemma (Lemma \ref{lemma53}), we can prove Theorem \ref{theoimersaoinfinito} for the first derivative case which is equivalent to the following proposition.

\begin{prop}\label{propimersaoinfinito}
Let $X^{1,p}_\infty(\alpha_0,\alpha_1)$ the weighted Sobolev space such that $p\geq1$ and $\alpha_0\geq\alpha_1-p$. Assume $\alpha_1-p\leq\theta\leq \alpha_0$.
\begin{flushleft}
    $\mathrm{(a)}$ Sobolev case: If $\alpha_1-p+1>0$, then the following continuous embedding holds:
    \begin{equation*}
    X^{1,p}_\infty\hookrightarrow L^q_\theta,\quad p\leq q\leq p^*,
    \end{equation*}
    where
    \begin{equation*}
    p^*:=\dfrac{(\theta+1)p}{\alpha_1-p+1}.
    \end{equation*}
    Moreover, it is a compact embedding if one of the following two conditions is fulfilled:
    \begin{flushleft}
    $\mathrm{(i)}$ $\theta=\alpha_0$ and $p<q<p^*$;\\
    $\mathrm{(ii)}$ $\theta<\alpha_0$ and $p\leq q<p^*$.
    \end{flushleft}
    $\mathrm{(b)}$ Sobolev Limit case: If $\alpha_1-p+1=0$, then the following continuous embedding holds:
    \begin{equation*}
    X^{1,p}_\infty\hookrightarrow L^q_\theta,\quad p\leq q<\infty.
    \end{equation*}
    Moreover, it is compact embedding if one of the following two conditions is fulfilled:
    \begin{flushleft}
    $\mathrm{(i)}$ $\theta=\alpha_0$, $q>p>1$, and $\alpha_0>-1$;\\
    $\mathrm{(ii)}$ $\theta<\alpha_0$ and $q\geq p\geq1$.
    \end{flushleft}
\end{flushleft}
\end{prop}
\begin{proof}
We first prove the continuous embeddings.
Let $q\geq p$ and $R\in(1,\infty)$ be fixed. Given $r\geq R$, by Lemma \ref{lemma53} we have
\begin{equation*}
|u(r)|^{q-p}\leq \dfrac{C}{r^{(q-p)\frac{\alpha_0(p-1)+\alpha_1}{p^2}}}\|u\|_{X^{1,p}_\infty}^{q-p}\leq \dfrac{C}{R^{(q-p)\frac{\alpha_0(p-1)+\alpha_1}{p^2}}}\|u\|_{X^{1,p}_\infty}^{q-p}.
\end{equation*}
Multiplying both sides by $|u(r)|^pr^\theta$ and then integrating over $(R,\infty)$,
\begin{equation}\label{eqinfinito}
\int_R^\infty|u|^qr^\theta\leq\dfrac{C}{R^{(q-p)\frac{\alpha_0(p-1)+\alpha_1}{p^2}+\alpha_0-\theta}}\|u\|^{q-p}_{X^{1,p}_\infty}\int_R^\infty|u|^pr^{\alpha_0}\mathrm dr.
\end{equation}
By Theorem \ref{theo32} (here we used $\theta\geq\alpha_1-p$) we get $\|u\|_{L^q_\theta(0,R)}\leq C\|u\|_{X^{1,p}_\infty}$ and from \eqref{eqinfinito} we conclude the continuous embedding.

We are left with the task of proving the compact embeddings.
We only need to show that $u_n\to0$ in $L^q_\theta(0,\infty)$ for any $u_n\rightharpoonup0$ in $X^{1,p}_\infty$. We claim that
\begin{equation}\label{eqq0}
(q-p)\dfrac{\alpha_0(p-1)+\alpha_1}{p^2}+\alpha_0-\theta>0.
\end{equation}
Indeed, we can suppose $\theta=\alpha_0$ since the case $\theta<\alpha_0$ is trivial. We state that $[\alpha_0(p-1)+\alpha_1]/p^2>0$ in both cases (Sobolev and Sobolev Limit). In fact, in the case $\alpha_1-p+1>0$, by $\alpha_0\geq\alpha_1-p$, we have
\begin{equation*}
\dfrac{\alpha_0(p-1)+\alpha_1}{p^2}\geq\dfrac{\alpha_1-p+1}p>0.
\end{equation*}
On the other hand, if $\alpha_1-p+1=0$, $p>1$ and $\alpha_0>-1$, then
\begin{equation*}
\dfrac{\alpha_0(p-1)+\alpha_1}{p^2}=\dfrac{(\alpha_0+1)(p-1)}{p^2}>0,
\end{equation*}
which concludes \eqref{eqq0}. By \eqref{eqinfinito}, \eqref{eqq0} and since $(u_n)$ is bounded in $X^{1,p}_\infty$, given $\varepsilon>0$ there exists $R_0>0$ such that
\begin{equation*}
\int_{R_0}^\infty|u_n|^qr^\theta\mathrm dr<\dfrac{\varepsilon}2,\quad\forall n\in\mathbb N.
\end{equation*}
From compact embedding of Theorem \ref{theo32}, there exists $n_0\in\mathbb N$ with
\begin{equation*}
\int_0^{R_0}|u_n|^qr^\theta\mathrm dr<\dfrac{\varepsilon}2,\quad\forall n\geq n_0.
\end{equation*}
Therefore $u_n\to0$ in $L^q_\theta(0,\infty)$.
\end{proof}

\begin{proof}[Proof of Theorem \ref{theoimersaoinfinito}]
The proof follows by induction on $k$. The case $k=1$ follows from Proposition \ref{propimersaoinfinito}. By  $\alpha_j\geq\alpha_k-(k-j)p$ for all $j=1,\ldots,k-1$ and $\alpha_k-(k-1)p+1>0$ we have the following continuous embedding
\begin{equation*}
u'\in X^{k-1,p}_\infty(\alpha_1,\ldots,\alpha_k)\hookrightarrow L^p_{\alpha_k-(k-1)p}.
\end{equation*}
Thus, $u\in X^{1,p}_\infty(\alpha_0,\alpha_k-(k-1)p)$. Using Proposition \ref{propimersaoinfinito}, given $\alpha_k-kp\leq\theta\leq\alpha_0$,
\begin{equation}\label{pik1}
X^{1,p}_\infty(\alpha_0,\alpha_k-(k-1)p)\hookrightarrow L^q_\theta,
\end{equation}
where $p\leq q\leq p^*$ in Sobolev case and $p\leq q<\infty$ in Sobolev Limit case.

For the compact embedding, we can see that (in both Sobolev and Sobolev Limit cases) \eqref{pik1} is a compact embedding under conditions $(i)$ or $(ii)$. Therefore, we obtain the compact embedding $X^{k,p}_\infty\hookrightarrow L^q_\theta$ from the composition of continuous embedding with a compact embedding.
\end{proof}

Before we prove Theorem \ref{theo0}, we need the following lemma which provides estimates on the term $\int_R^\infty\exp_p(\mu|u|^{p'})r^\theta\mathrm dr$.

\begin{lemma}\label{lemma75}
Let $X^{1,p}_\infty(\alpha_0,p-1)$ be the weighted Sobolev space such that $\alpha_0>-1$ and $p>1$. If $R\in(0,\infty)$, $\mu>0$, $\theta\in[-1,\alpha_0]$ and $u\in X^{1,p}_\infty$, then $\exp_p(\mu|u|^{p'})\in L^1_\theta(R,\infty)$. Moreover, if $\|u\|_{X^{1,p}_\infty}\leq M$ then $\int_R^\infty \exp_p(\mu|u|^{p'})r^\theta\mathrm dr\leq C$, where $C$ depends only on $\alpha_0,p,R,\mu,\theta$ and $M$.
\end{lemma}
\begin{proof}
Our proof is based in \cite[Lemma 3.2]{MR3670473}. Let $u\in X^{1,p}_\infty(\alpha_0,p-1)$ such that $\|u\|_{X^{1,p}_\infty}\leq M$. From Proposition \ref{propimersaoinfinito}, we have $u\in L^q_\theta(0,\infty)$ for any $q\in[p,\infty)$. The monotone convergence theorem ensures
\begin{align*}
\int_R^\infty \exp_p(\mu|u|^{p'})r^\theta\mathrm dr&=\sum_{j=0}^\infty\dfrac{\mu^{p-1+j}}{\Gamma(p+j)}\int_R^{\infty}|u|^{(p-1+j)p'}r^\theta\mathrm dr\\
&=\dfrac{\mu^{p-1}}{\Gamma(p)}\|u\|_{L^{p}_\theta}^{p}+\dfrac{\mu^{p}}{\Gamma(p+1)}\|u\|_{L^{pp'}_\theta}^{pp'}+\sum_{j=2}^\infty\dfrac{\mu^{p-1+j}}{\Gamma(p+j)}\int_R^{\infty}|u|^{(p-1+j)p'}r^\theta\mathrm dr.
\end{align*}
Note that $pp'\geq p$. Then, Proposition \ref{propimersaoinfinito} implies $\|u\|_{L^q_\theta}\leq C_q \|u\|_{X^{1,p}_\infty}$ for each $q\in[p,\infty)$.
\begin{align}
\int_R^\infty \exp_p(\mu|u|^{p'})r^\theta\mathrm dr&\leq C_p^p\dfrac{\mu^{p-1}}{\Gamma(p)}\|u\|_{X^{1,p}_\infty}^{p}+C_{pp'}^{pp'}\dfrac{\mu^{p}}{\Gamma(p+1)}\|u\|_{X^{1,p}_\infty}^{pp'}\nonumber\\
&\quad+\sum_{j=2}^\infty\dfrac{\mu^{p-1+j}}{\Gamma(p+j)}\int_R^{\infty}|u|^{(p-1+j)p'}r^\theta\mathrm dr.\label{eq45}
\end{align}
Now it is enough to estimate the last term in \eqref{eq45}. For this, we use to Lemma \ref{lemma53} in the following way
\begin{equation*}
\int_R^\infty |u|^{(p-1+j)p'}r^\theta\mathrm dr\leq C^{(p-1+j)p'}\|u\|_{X^{1,p}_\infty}^{(p-1+j)p'}\int_R^\infty r^{\theta-(p-1+j)\frac{\alpha_0+1}{p}}\mathrm dr.
\end{equation*}
Since $p-1+j\geq p+1$ and $r\geq R$ we have \begin{equation*}r^{\theta-(p-1+j)\frac{\alpha_0+1}p}\leq R^{(j-2)\frac{\alpha_0+1}{p}}r^{\theta-(p+1)\frac{\alpha_0+1}p}.\end{equation*} Thus,
\begin{align*}
\int_R^\infty |u|^{(p-1+j)p'}r^\theta\mathrm dr&\leq C^{(p-1+j)p'}\|u\|_{X^{1,p}_\infty}^{(p-1+j)p'}R^{(j-2)\frac{\alpha_0+1}{p}}\int_R^{\infty}r^{\theta-(p+1)\frac{\alpha_0+1}p}\mathrm dr\\
&=\dfrac{pC^{(p-1+j)p'}R^{(j-p-3)\frac{\alpha_0+1}p+\theta+1}}{(p+1)(\alpha_0+1)-p\theta-p}\|u\|_{X^{1,p}_\infty}^{(p-1+j)p'}.
\end{align*}
Using in \eqref{eq45} we have
\begin{align*}
\int_R^\infty \exp_p(\mu|u|^{p'})r^\theta\mathrm dr&\leq C_p^p\dfrac{\mu^{p-1}}{\Gamma(p)}\|u\|_{X^{1,p}_\infty}^{p}+C_{pp'}^{pp'}\dfrac{\mu^{p}}{\Gamma(p+1)}\|u\|_{X^{1,p}_\infty}^{pp'}\nonumber\\
&\quad+\dfrac{pR^{\theta+1-2(p+1)\frac{\alpha_0+1}{p}}}{(p+1)(\alpha_0+1)-p\theta-p}\sum_{j=2}^\infty\dfrac{\left(\mu C^{p'}R^{\frac{\alpha_0+1}p}\|u\|_{X^{1,p}_\infty}^{p'}\right)^{p-1+j}}{\Gamma(p+j)},
\end{align*}
which gives our estimate.
\end{proof}

\begin{proof}[Proof of Theorem \ref{theo0}]
Splitting the integral in the following way
\begin{equation*}
\int_0^\infty\exp_p(\mu|u|^{p'})r^{\theta}\mathrm dr=\int_0^1\exp_p(\mu|u|^{p'})r^{\theta}\mathrm dr+\int_1^\infty\exp_p(\mu|u|^{p'})r^{\theta}\mathrm dr
\end{equation*}
and using Theorem \ref{theo01} with Lemma \ref{lemma75} we conclude our Theorem \ref{theo0}. Note that the sequence in the proof of item (c) in Theorem \ref{theo01} belongs to $X^{1,p}_{0,1}$.
\end{proof}

\section{Proof of Theorem \ref{theo12}}\label{sec3}

Initially, we should evaluate the behavior of $\exp_p$ when $p$ changes. This is important for the Remark \ref{remark1} and for the estimate $\exp_p(t)\leq \mathrm{e}^t$. The following lemma guarantees this behavior.

\begin{lemma}\label{lemmaphiN}
Let $p\in [1,\infty)$. Then
\begin{equation}\label{stack}
\exp_p(t)=\mathrm{e}^t-\mathrm{e}^t\dfrac{p-1}{\Gamma(p)}\int_t^\infty s^{p-2}\mathrm{e}^{-s} \mathrm ds, \quad\forall t\in(0,\infty).
\end{equation}
Given $1\leq q<p$ we have $\exp_p(t)<\exp_q(t)$ for all $t\in(0,\infty)$. In particular, $\exp_p(t)\leq \mathrm{e}^t$ for all $t\in[0,\infty)$.
\end{lemma}
\begin{proof}
Firstly, let us prove \eqref{stack}. For $p=1$ is trivial. If $p>1$ we define $f_p\colon [0,\infty)\to\mathbb R$ given by
\begin{equation*}
f_p(t)=\exp_p(t)-\mathrm{e}^t+\mathrm{e}^t\dfrac{p-1}{\Gamma(p)}\int_t^\infty s^{p-2}\mathrm{e}^{-s}\mathrm ds.
\end{equation*}
Note that
\begin{equation*}
f_p(0)=-1+\dfrac{p-1}{\Gamma(p)}\Gamma(p-1)=0
\end{equation*}
and
\begin{align*}
f_p'(t)&=\dfrac{p-1}{\Gamma(p)}t^{p-2}+\exp_p(t)-\mathrm{e}^t+\dfrac{p-1}{\Gamma(p)}\left[\mathrm{e}^t\int_t^\infty s^{p-2}\mathrm{e}^{-s}\mathrm ds-\mathrm{e}^tt^{p-2}\mathrm{e}^{-t}\right]\\
&=\exp_p(t)-\mathrm{e}^t+\mathrm{e}^t\dfrac{p-1}{\Gamma(p)}\int_t^\infty s^{p-2}\mathrm{e}^{-s}\mathrm ds\\
&=f_p(t).
\end{align*}
Hence $f_p\equiv0$, which concludes \eqref{stack}.

Now given $1\leq q<p$, we define $g_{p,q}\colon(0,\infty)\to\mathbb R$ by
\begin{equation*}
g_{p,q}(t)=\dfrac{\exp_p(t)}{\exp_q(t)}.
\end{equation*}
We have  $g_{p,q}(t)>1$ as an immediate consequence of
\begin{enumerate}
    \item[(i)] $\lim_{t\to0}g_{p,q}(t)=0$;
    \item[(ii)] $\lim_{t\to\infty}g_{p,q}(t)=1$;
    \item[(iii)] $g_{p,q}$ is increasing.
\end{enumerate}
To deal with (i) we note that
\begin{equation*}
g_{p,q}(t)=\dfrac{\frac{t^{p-1}}{\Gamma(p)}+\frac{t^p}{\Gamma(p+1)}+\cdots}{\frac{t^{q-1}}{\Gamma(q)}+\frac{t^q}{\Gamma(q+1)}+\cdots}=\dfrac{\frac{t^{p-q}}{\Gamma(p)}+\frac{t^{p-q+1}}{\Gamma(p+1)}+\cdots}{\frac{1}{\Gamma(q)}+\frac{t}{\Gamma(q+1)}+\cdots}\overset{t\to0}\longrightarrow 0.
\end{equation*}
On the other hand, (ii) follows by
\begin{equation*}
\lim_{t\to\infty}g_{p,q}(t)=\lim_{t\to\infty}\dfrac{1-\dfrac{p-1}{\Gamma(p)}\int_t^\infty s^{p-2}\mathrm{e}^{-s}\mathrm ds}{1-\dfrac{q-1}{\Gamma(q)}\int_t^\infty s^{q-2}\mathrm{e}^{-s}\mathrm ds}=1,
\end{equation*}
which is clear from \eqref{stack}. Now we are left with the task to check $g_{p,M}'(t)>0$. Using $\Gamma(p+j)=\Gamma(p)\prod_{i=0}^{j-1}(p+i)$ we have
\begin{align*}
g'_{N,q}(t)\exp_q(t)^2&=\dfrac{p-1}{\Gamma(p)}t^{p-2}\exp_q(t)-\dfrac{q-1}{\Gamma(q)}t^{q-2}\exp_p(t)\\
&=\sum_{j=0}^\infty\left(\dfrac{p-1}{\Gamma(p)\Gamma(q+j)}-\dfrac{q-1}{\Gamma(q)\Gamma(p+j)}\right)t^{p+q-3}\\
&=\sum_{j=0}^\infty\dfrac{(p-1)\prod_{i=0}^{j-1}(p+i)-(q-1)\prod_{i=0}^{j-1}(q+j)}{\Gamma(q+j)\Gamma(p+j)}t^{p+q-3}\\
&>0,
\end{align*}
which completes the proof of the lemma.
\end{proof}

We introduce the notation used by M. Ishiwata \cite{MR2854113} that allows us to analyze the lack of compactness at infinite. (See also \cite{CLZ-CVPDE} for nonradial case.) Fixed $u_n\rightharpoonup u$ in $X^{1,p}_{\infty}(\alpha_0,p-1)$ with $\|u_n\|_{X^{1,p}_{\infty}}=1$ and $\mu>0$, define
\begin{equation*}
\rho_0=\lim_{R\to\infty}\lim_{n\to\infty}\int_0^R|u_n'|^pr^{p-1}+|u_n|^pr^{\alpha_0}\mathrm dr,
\end{equation*}
\begin{equation*}
\rho_\infty=\lim_{R\to\infty}\lim_{n\to\infty}\int_R^\infty|u_n'|^pr^{p-1}+|u_n|^pr^{\alpha_0}\mathrm dr,
\end{equation*}
\begin{equation*}
\nu_0=\lim_{R\to\infty}\lim_{n\to\infty}\int_0^R\exp_p(\mu |u_n|^{p'})r^{\alpha_0}\mathrm dr,
\end{equation*}
\begin{equation*}
\nu_\infty=\lim_{R\to\infty}\lim_{n\to\infty}\int_R^\infty\exp_p(\mu|u_n|^{p'})r^{\alpha_0}\mathrm dr,
\end{equation*}
\begin{equation*}
\eta_0=\lim_{R\to\infty}\lim_{n\to\infty}\int_0^R|u_n|^pr^{\alpha_0}\mathrm dr,
\end{equation*}
\begin{equation*}
\eta_\infty=\lim_{R\to\infty}\lim_{n\to\infty}\int_R^\infty|u_n|^pr^{\alpha_0}\mathrm dr,
\end{equation*}
up to subsequence if necessary.

For a sufficiently large number $R$, take a cut-off function $\varphi_R^0\in C^\infty([0,\infty))$ such that
\begin{equation*}
\left\{\begin{array}{llll}
\varphi_R^0(r)=1&\mbox{for }0\leq r<R,\\
0\leq \varphi_R^0(r)\leq1&\mbox{for }R\leq r<R+1,\\
\varphi_R^0(r)=0&\mbox{for }R+1\leq r,\\
|(\varphi_R^0)'(r)|\leq2&\mbox{for all }r
\end{array}\right.
\end{equation*}
and introduce another cut-off function $\varphi_R^\infty$ by
\begin{equation*}
\varphi_R^\infty(r):=1-\varphi_R^0(r).
\end{equation*}

Throughout the rest of this section, all the results are based in \cite{MR4097244} and \cite{MR2854113}.

\begin{lemma}\label{il21}
Let $u_{n,R}^*:=u_n\varphi_R^*$ ($*=0,\infty$). We have
\begin{equation*}
\rho_*=\lim_{R\to\infty}\lim_{n\to\infty}\int_0^\infty|(u_{n,R}^*)'|^pr^{p-1}+|u_{n,R}^*|^pr^{\alpha_0}\mathrm dr,
\end{equation*}
\begin{equation*}
\nu_*=\lim_{R\to\infty}\lim_{n\to\infty}\int_0^\infty\exp_p(\mu|u_{n,R}^*|^{p'})r^{\alpha_0}\mathrm dr,
\end{equation*}
\begin{equation*}
\eta_*=\lim_{R\to\infty}\lim_{n\to\infty}\int_0^\infty|u^*_{n,R}|^pr^{\alpha_0}\mathrm dr.
\end{equation*}
\end{lemma}
\begin{proof}
Since, for $F(t)=|t|^p$ or $F(t)=\exp_p(\mu|t|^{p'})$, we have
\begin{equation*}
\int_0^RF(u_n)r^{\alpha_0}\mathrm dr\leq\int_0^\infty F(u_{n,R}^0)r^{\alpha_0}\mathrm dr\leq\int_0^{R+1}F(u_n)r^{\alpha_0}\mathrm dr
\end{equation*}
and
\begin{equation*}
\int_{R+1}^\infty F(u_n)r^{\alpha_0}\mathrm dr\leq\int_0^\infty F(u_{n,R}^\infty)r^{\alpha_0}\mathrm dr\leq\int_R^\infty F(u_n)r^{\alpha_0}\mathrm dr,
\end{equation*}
it is enough to check that
\begin{equation}\label{dente1}
\lim_{R\to\infty}\lim_{n\to\infty}\int_0^\infty|(u_{n,R}^0)'|^pr^{p-1}\mathrm dr=\lim_{R\to\infty}\lim_{n\to\infty}\int_0^R|u_n'|^pr^{p-1}\mathrm dr
\end{equation}
and
\begin{equation}\label{dente2}
\lim_{R\to\infty}\lim_{n\to\infty}\int_0^\infty|(u_{n,R}^\infty)'|^pr^{p-1}\mathrm dr=\lim_{R\to\infty}\lim_{n\to\infty}\int_R^\infty|u_n'|^pr^{p-1}\mathrm dr.
\end{equation}
We will just prove \eqref{dente1} since the proof of \eqref{dente2} is analogous. Note that, by Mean Value Theorem,
\begin{align}
\int_0^\infty\left|(u_{n,R}^0)'\right|^pr^{p-1}\mathrm dr&=\int_0^\infty\left|\varphi^0_Ru_n'+(\varphi_R^0)'u_n\right|^pr^{p-1}\mathrm dr\nonumber\\
&=\int_0^\infty|u_n'|^p|\varphi_R^0|^pr^{p-1}\mathrm dr+\mathcal R_{n,R},\label{dente3}
\end{align}
where $0<\theta<1$ and
\begin{equation*}
\mathcal R_{n,R}=p\int_0^\infty\left|\varphi_R^0u_n'+\theta(\varphi_R^0)'u_n\right|^{p-2}\left(\varphi_R^0u_n'+\theta(\varphi_R^0)'u_n\right)(\varphi_R^0)'u_nr^{p-1}\mathrm dr.
\end{equation*}
By Theorem \ref{theo32} we have
\begin{equation*}
\lim_{n\to\infty}\int_R^{R+1}|u_n|^pr^{p-1}\mathrm dr=\int_R^{R+1}|u|^pr^{p-1}\mathrm dr.
\end{equation*}
Therefore,
\begin{equation}\label{dente4}
\lim_{R\to\infty}\lim_{n\to\infty}\mathcal R_{n,R}=0.
\end{equation}
Using \eqref{dente3}, \eqref{dente4} and
\begin{equation*}
\int_0^R|u_n'|^pr^{p-1}\mathrm dr\leq\int_0^\infty|u_n'|^p|\varphi^0_R|^pr^{p-1}\mathrm dr\leq\int_0^{R+1}|u_n'|^pr^{p-1}\mathrm dr
\end{equation*}
we conclude \eqref{dente1}.
\end{proof}

For the proof of Theorem \ref{theo12}, we need to study $d_{p,\mu}$ for a vanishing sequence $(u_n)$. Here is the precise definition of vanishing sequence.


\begin{definition}
A sequence $(u_n)\subset X^{1,p}_{\infty}(\alpha_0,p-1)$ is a \textit{normalized vanishing sequence} ((NVS) in short) if
\begin{enumerate}
    \item $\|u_n\|_{X^{1,p}_{\infty}}=1$,
    \item $u_n\rightharpoonup 0$ weakly in $X^{1,p}_{\infty}$,
    \item $\nu_0=\lim_{R\to\infty}\lim_{n\to\infty}\int_R^\infty\exp_p(\mu|u_n|^{p'})r^{\alpha_0}\mathrm dr=0$.
\end{enumerate}
\end{definition}

\begin{definition}\begin{flushleft}
   The number
    \begin{equation*}
    d_{NVL}(p,\mu)=\sup_{(u_n)\colon(NVS)}\int_0^\infty\exp_p\left(\mu|u_n|^{p'}\right)r^{\alpha_0}\mathrm dr
    \end{equation*}
    is called a \textit{normalized vanishing limit}.
\end{flushleft}
\end{definition}

\begin{prop}\label{ip21}
There holds $d_{NVL}(p,\mu)=\mu^{p-1}/\Gamma(p)$ if $\alpha_0>-1$.
\end{prop}
\begin{proof}
Let $(u_n)$ be an (NVS). Up to subsequence,
\begin{equation}
\nu_\infty=\lim_{R\to\infty}\lim_{n\to\infty}\int_R^\infty\left(\exp_p(\mu|u_n|^{p'})-\dfrac{\mu^{p-1}}{\Gamma(p)}|u_n|^{p}\right)r^{\alpha_0}\mathrm dr+\lim_{R\to\infty}\lim_{n\to\infty}\dfrac{\mu^{p-1}}{\Gamma(p)}\|u_n\|^{p}_{L^{p}_{\alpha_0}(R,\infty)}.\label{ish13}
\end{equation}
We first prove that
\begin{equation}\label{dente5}
\lim_{R\to\infty}\lim_{n\to\infty}\int_R^\infty\left(\exp_p(\mu|u_n|^{p'})-\dfrac{\mu^{p-1}}{\Gamma(p)}|u_n|^{p}\right)r^{\alpha_0}\mathrm dr=0.
\end{equation}
By Lemma \ref{lemma53},
\begin{equation}\label{dente7}
|u_n(r)|\leq C\|u_n\|_{X^{1,p}_\infty}\dfrac1{r^{\frac{(\alpha_0+1)(p-1)}{p^2}}}\quad\forall r\in(0,\infty),
\end{equation}
where $C>0$ depends only on $p$. This gives, for $j\geq2$,
\begin{equation*}
\|u_n\|_{L^{p'(p-1+j)}_{\alpha_0}(R,\infty)}^{p'(p-1+j)}\leq C^{p-1+j}\int_R^\infty r^{\alpha_0-\frac{(\alpha_0+1)(p-1+j)}{p}}\mathrm dr\leq\dfrac{pC^{p-1+j}}{\alpha_0+1}R^{\frac{(\alpha_0+1)(1-j)}{p}}.
\end{equation*}
Then, for $R\geq1$,
\begin{align}
\sum_{j=2}^\infty\dfrac{\mu^{p-1+j}}{\Gamma(p+j)}\|u_n\|^{p'(p-1+j)}_{L^{p'(p-1+j)}_{\alpha_0}(R,\infty)}&\leq\dfrac{p}{\alpha_0+1}\sum_{j=2}^\infty\dfrac{(\mu C)^{p-1+j}}{\Gamma(p+j)}R^{\frac{(\alpha_0+1)(1-j)}{p}}\nonumber\\
&=\dfrac{p}{\alpha_0+1}(\mu C)^{p}\sum_{j=2}^\infty\dfrac{(\mu CR^{-\frac{\alpha_0+1}{p}})^{j-1}}{\Gamma(p+j)}\nonumber\\
&\leq\dfrac{p}{\alpha_0+1}(\mu C)^{p}\sum_{j=2}^\infty\dfrac{(\mu CR^{-\frac{\alpha_0+1}p})^{j-1}}{(j-1)!}\nonumber\\
&=\dfrac{p}{\alpha_0+1}(\mu C)^p\left(\mathrm{e}^{\mu CR^{-\frac{\alpha_0+1}{p}}}-1\right).\label{dente6}
\end{align}

We claim that for each $\gamma>p$ and $\eta>1$ with $p-\eta(p^2-\gamma p+\gamma)>0$ (for $\gamma\geq p^2/(p-1)$ take any $\eta>1$ and for $\gamma<p^2/(p-1)$ take any $1<\eta<p/(p^2-\gamma p+\gamma)$) there exists $\widetilde C>0$ depending only on $p$, $\gamma$ and $\eta$ such that
\begin{equation}\label{dente8}
\|u_n\|_{L^\gamma_{\alpha_0}(R,\infty)}^\gamma\leq \widetilde CR^{-\frac{(\alpha_0+1)[p-\eta(p^2-\gamma p+\gamma)]}{p^2\eta}}.
\end{equation}
Indeed, note that $\gamma=\frac{p}\eta+\frac{\eta\gamma-p}{\eta}$. By H\"older inequality and \eqref{dente7} we obtain
\begin{align*}
\|u_n\|^{\gamma}_{L^{\gamma}_{\alpha_0}(R,\infty)}&\leq\left(\int_R^\infty|u_n|^pr^{\alpha_0}\mathrm dr\right)^{\frac1\eta}\left(\int_R^\infty|u_n|^{\frac{\eta\gamma-p}{\eta-1}}r^{\alpha_0}\mathrm dr\right)^{\frac{\eta-1}\eta}\\
&\leq\sup_{n\in\mathbb N}\|u_n\|_{L^p_{\alpha_0}(0,\infty)}^{\frac p\eta}C^{\frac{\eta-1}\eta}\left(\int_R^\infty r^{\alpha_0-\frac{(\alpha_0+1)(p-1)(\eta\gamma-p)}{p^2(\eta-1)}}\mathrm dr\right)^{\frac{\eta-1}\eta}\\
&\leq C^{\frac{\eta-1}\eta}\left(\dfrac{p^2(\eta-1)}{(\alpha_0+1)[p-\eta(p^2-\gamma p+\gamma)]}\right)^{\frac{\eta-1}\eta}R^{-\frac{(\alpha_0+1)[p-\eta(p^2-\gamma p+\gamma)]}{p^2\eta}}.
\end{align*}
This concludes \eqref{dente8}. Since the right term of \eqref{dente6} and \eqref{dente8} (with $\gamma=p'p>p$) tend to 0 as $R\to\infty$ uniformly in $n$, we have \eqref{dente5}.

Using $\nu_0=0$ (because $(u_n)$ is an (NVS)) together with $\|u_n\|_{L^p_{\alpha_0}(0,\infty)}^p\leq \|u_n\|^p_{X^{1,p}_{\infty}}\leq 1$, we get
\begin{align*}
\lim_{n\to\infty}\int_0^\infty\exp_p(\mu|u_n|^{p'})r^{\alpha_0}\mathrm dr&=\nu_0+\nu_\infty=\nu_\infty\\
&=\lim_{R\to\infty}\lim_{n\to\infty}\dfrac{\mu^{p-1}}{\Gamma(p)}\|u_n\|^p_{L^p_{\alpha_0}(R,\infty)}\leq\dfrac{\mu^{p-1}}{\Gamma(p)}.
\end{align*}

It remains to prove the converse inequality: $d_{NVL}(p,\mu)\geq \mu^{p-1}/\Gamma(p)$. Let $\psi_0\colon[0,\infty)\to\mathbb R$ be a smooth function satisfying $\|\psi'\|_{L^p_{p-1}}=\|\psi\|_{L^p_{\alpha_0}}=1$ with $\psi(r)=0$ for $r$ sufficiently large. Fixed $(\gamma_n)$ a sequence in $(0,\infty)$ with $\gamma_n\to0$, we define $\psi_n(r)=\gamma_n\psi_0(\gamma_n^{\frac{p}{\alpha_0+1}}r)$. It follows easily that
\begin{equation}\label{dente11}
\|\psi'\|_{L^p_{p-1}}=\gamma_n,\ \|\psi\|_{L^p_{\alpha_0}}=1,\ \|\psi_n\|_{X^{1,p}_{\infty}}\to1,\ \lim_{R\to\infty}\lim_{n\to\infty}\int_R^\infty|\psi_n|^pr^{\alpha_0}\mathrm dr=1.
\end{equation}
We claim that $(\psi_n/\|\psi_n\|_{X^{1,p}_\infty})$ is a (NVS). Indeed, fixed $R>0$ note that
\begin{equation*}
\int_0^R\exp_p\left(\mu\left|\dfrac{\psi_n(r)}{\|\psi_n\|_{X^{1,p}_{\infty}}}\right|^{p'}\right)r^{\alpha_0}\mathrm dr\leq\exp_p\left(\dfrac{\mu\gamma_n^{p}\|\psi_0\|_{L^\infty}^{p'}}{\|\psi_n\|^{p'}_{X^{1,p}_\infty}}\right)\int_0^Rr^{\alpha_0}\mathrm dr\overset{n\to\infty}\longrightarrow0.
\end{equation*}
Then $\nu_0=0$ for $(\psi_n/\|\psi_n\|_{X^{1,p}_\infty})$. Since $\psi_n(r)\to0$ for each $r\in[0,\infty)$ we conclude that $(\psi_n/\|\psi_n\|_{X^{1,p}_\infty})$ is a (NVS). By $\nu_0=0$, \eqref{dente11} and \eqref{dente5}, we have
\begin{align*}
d_{NVL}(p,\mu)&\geq\lim_{n\to\infty}\int_0^\infty\exp_p\left(\mu\left|\dfrac{\psi_n}{\|\psi_n\|_{X^{1,p}_\infty}}\right|^{p'}\right)r^{\alpha_0}\mathrm dr\\
&=\nu_0+\nu_\infty=\nu_\infty\\
&=\lim_{R\to\infty}\lim_{n\to\infty}\int_R^\infty\exp_p\left(\dfrac{\mu|\psi_n|^{p'}}{\|\psi_n\|_{X^{1,p}_\infty}^{p'}}\right)r^{\alpha_0}\mathrm dr\\
&=\lim_{R\to\infty}\lim_{n\to\infty}\dfrac{\mu^{p-1}}{\Gamma(p)}\dfrac{\int_R^\infty|\psi_n|^pr^{\alpha_0}\mathrm dr}{\|\psi_n\|^p_{X^{1,p}_\infty}}\\
&\quad+\lim_{R\to\infty}\lim_{n\to\infty}\int_R^\infty\left(\exp_p(\mu|\psi_n|^{p'})-\dfrac{\mu^{p-1}}{\Gamma(p)}|\psi_n|^{p}\right)r^{\alpha_0}\mathrm dr\\
&=\dfrac{\mu^{p-1}}{\Gamma(p)}.
\end{align*}
\end{proof}

\begin{prop}\label{ip22}
Assume $\alpha_0>-1$. Then
\begin{equation*}
    d_{p,\mu}>\frac{\mu^{p-1}}{\Gamma(p)}
\end{equation*}
provided that one of the following conditions holds:
\begin{flushleft}
    \noindent$\mathrm{(i)}$ $p>2$ and $\mu\in(0,\infty)$;\\
    \noindent$\mathrm{(ii)}$ $p=2$ and $\mu\in(2/B_{2,\alpha_0},\infty)$,
    where $B_{2,\alpha_0}=\sup_{u\in X^{1,2}_\infty\backslash\{0\}}\frac{\|u\|^4_{L^4_{\alpha_0}}}{\|u'\|_{L^2_1}^2\|u\|_{L^2_{\alpha_0}}^2}$.
\end{flushleft}
\end{prop}
\begin{proof}
Fixed $v\in X^{1,p}_\infty\backslash\{0\}$, we define the family of functions $v_t$ given by
\begin{equation*}
v_t(r):=t^{\frac1p}v\left(t^{\frac1{\alpha_0+1}}r\right),
\end{equation*}
where $t>0$ is a parameter. For $q\geq1$, we can see that
\begin{equation*}
\|v'_t\|_{L^{p}_{p-1}}=t^{\frac1p}\|v'\|_{L^p_{p-1}}\mbox{ and }\|v_t\|_{L^q_{\alpha_0}}=t^{\frac{1}p-\frac{1}{q}}\|v\|_{L^q_{\alpha_0}}.
\end{equation*}
Then
\begin{align*}
d_{p,\mu}&\geq\int_0^\infty\exp_p\left[\mu\left(\dfrac{|v_t|}{\|v_t\|_{X^{1,p}_\infty}}\right)^{p'}\right]r^{\alpha_0}\mathrm dr\\
&=\sum_{j=0}^\infty\dfrac{\mu^{p-1+j}}{\Gamma(p+j)}\dfrac{\|v_t\|^{(p-1+j)p'}_{L^{(p-1+j)p'}_{\alpha_0}}}{\left(\|v_t'\|^p_{L^p_{p-1}}+\|v_t\|^p_{L^p_{\alpha_0}}\right)^{\frac{p-1+j}{p-1}}}\\
&\geq\dfrac{\mu^{p-1}}{\Gamma(p)}\left[\dfrac{\|v\|_{L^p_{\alpha_0}}^p}{t\|v'\|^p_{L^{p}_{p-1}}+\|v\|^p_{L^p_{\alpha_0}}}+\dfrac{\mu}{p}\dfrac{t^{\frac1{p-1}}\|v\|^{pp'}_{L_{\alpha_0}^{pp'}}}{(t\|v'\|^p_{L^p_{p-1}}+\|v\|^p_{L^p_{\alpha_0}})^{p'}}\right]\\
&=:\dfrac{\mu^{p-1}}{\Gamma(p)}f_{p,\mu,v}(t).
\end{align*}
Since $f_{p,\mu,v}(0)=1$, to conclude $d_{p,\mu}>\frac{\mu^{p-1}}{\Gamma(p)}$ we only need to show $f'_{p,\mu,v}(t)>0$ if $0<t\ll1$.

Calculating $f'_{p,\mu,v}$, we have
\begin{equation*}
f'_{p,\mu,v}(t)=\dfrac{\|v'\|_{L^p_{p-1}}^p\|v\|_{L^p_{\alpha_0}}^p}{(t\|v'\|^p_{L^p_{p-1}}+\|v\|^p_{L^p_{\alpha_0}})^{2+\frac{1}{p-1}}}\dfrac1{t^{\frac{p-2}{p-1}}}g_{p,\mu,v}(t),
\end{equation*}
where
\begin{align*}
g_{p,\mu,v}(t)&=-t^{\frac{p-2}{p-1}}\left(t\|v'\|_{L^p_{p-1}}^p+\|v\|_{L^p_{\alpha_0}}^p\right)^{\frac{1}{p-1}}\\
&\quad+\dfrac{\mu}{p(p-1)}\dfrac{\|v\|_{L^{pp'}_{\alpha_0}}^{pp'}}{\|v'\|^p_{L^p_{p-1}}\|v\|_{L^{p}_{\alpha_0}}^{p}}\left(-t(p-1)\|v'\|^p_{L^p_{p-1}}+\|v\|_{L^{p}_{\alpha_0}}^{p}\right).
\end{align*}
If we prove that $g_{p,\mu,v}(0)>0$, proposition follows by continuity of $g_{p,\mu,v}\colon[0,\infty)\to\mathbb R$.

Firstly, suppose $p>2$. Then, for any $v\in X^{1,p}_\infty\backslash\{0\}$, we have
\begin{equation*}
g_{p,\mu,v}(0)=\dfrac{\mu}{p(p-1)}\dfrac{\|v\|_{L^{pp'}_{\alpha_0}}^{pp'}}{\|v'\|^p_{L^p_{p-1}}}>0.
\end{equation*}

The case $p=2$ is more delicate. In this case, we demand $\mu>2/B_{2,\alpha_0}$. By Proposition \ref{lemmab2} (with $p=2,q=4$ and $\alpha_1=1$) there exists $u_0\in X^{1,2}_\infty$ such that
\begin{equation*}
B_{2,\alpha_0}=\dfrac{\|u_0\|^4_{L^4_{\alpha_0}}}{\|u_0'\|^2_{L^2_1}\|u_0\|^2_{L^2_{\alpha_0}}}.
\end{equation*}
Therefore,
\begin{equation*}
g_{2,\mu,u_0}(0)=-\|u_0\|_{L^2_{\alpha_0}}^2+\dfrac{\mu}{2}\dfrac{\|u_0\|_{L^{4}_{\alpha_0}}^4}{\|u_0'\|^2_{L^2_1}}=\|u_0\|_{L^2_{\alpha_0}}\left(-1+\dfrac{\mu}{2}B_{2,\alpha_0}\right)>0,
\end{equation*}
which is our claim.
\end{proof}

The following proposition generalizes \cite[Proposition 7.1]{MR4097244}, because we do not assume $p=2$, $q=4$, $\alpha_1=1$, and $\alpha_0\geq0$. Our proof is motivated by   \cite[Theorem B]{MR0691044}. (see also \cite{CLZ-AIM} in the nonradial case.)

\begin{prop}\label{lemmab2}
Let $X^{1,p}_\infty(\alpha_0,\alpha_1)$ be the weighted Sobolev space with $1\leq p<\infty$, $\alpha_1-p+1\geq0$ and $\alpha_0>\alpha_1-p$. Given $p\leq q<(\alpha_0+1)p/(\alpha_1-p+1)$, then the following supremum
\begin{equation*}
B_{p,\alpha_0}=\sup_{u\in X^{1,p}_\infty\backslash\{0\}}\dfrac{\|u\|^q_{L^q_{\alpha_0}}}{\|u'\|^{\frac{(\alpha_0+1)(q-p)}{\alpha_0+p-\alpha_1}}_{L^p_{\alpha_1}}\|u\|^{\frac{p(\alpha_0+1)-q(\alpha_1-p+1)}{\alpha_0+p-\alpha_1}}_{L^p_{\alpha_0}}}
\end{equation*}
is attained by a nonnegative function $u_0\in X^{1,p}_\infty$ such that $\|u_0\|_{L^p_{\alpha_0}}=\|u_0'\|_{L^p_{\alpha_1}}=1$.
\end{prop}
\begin{proof}
Define
\begin{equation*}
J(u)=\dfrac{\|u\|^q_{L^q_{\alpha_0}}}{\|u'\|^{\frac{(\alpha_0+1)(q-p)}{\alpha_0+p-\alpha_1}}_{L^p_{\alpha_1}}\|u\|^{\frac{p(\alpha_0+1)-q(\alpha_1-p+1)}{\alpha_0+p-\alpha_1}}_{L^p_{\alpha_0}}}.
\end{equation*}
Fixed $\mu,\lambda>0$, set $u_{\lambda,\mu}(r):=\lambda u(\mu r)$. We claim
\begin{enumerate}
    \item[(i)] $\|u_{\lambda,\mu}\|_{L^p_{\alpha_0}}=\lambda\mu^{-\frac{\alpha_0+1}{p}}\|u\|_{L^p_{\alpha_0}}$;
    \item[(ii)] $\|u_{\lambda,\mu}'\|_{L^p_{\alpha_1}}=\lambda\mu^{\frac{p-\alpha_1-1}{p}}\|u'\|_{L^p_{\alpha_1}}$;
    \item[(iii)] $J(u_{\lambda,\mu})=J(u)$.
\end{enumerate}
We are only going to prove (iii), since (i) and (ii) follow by an easy calculation:

\begin{align*}
J(u_{\lambda,\mu})&=\dfrac{\lambda^q\mu^{-\alpha_0-1}\|u\|^q_{L^q_{\alpha_0}}}{(\lambda \mu^{\frac{p-\alpha_1-1}{p}}\|u'\|_{L^p_{\alpha_1}})^{\frac{(\alpha_0+1)(q-p)}{\alpha_0+p-\alpha_1}}(\lambda\mu^{-\frac{\alpha_0+1}{p}}\|u\|_{L^p_{\alpha_0}})^{\frac{p(\alpha_0+1)-q(\alpha_1-p+1)}{\alpha_0+p-\alpha_1}}}\\
&=\lambda^{q-\frac{(\alpha_0+1)(q-p)}{\alpha_0+p-\alpha_1}-\frac{p(\alpha_0+1)-q(\alpha_1-p+1)}{\alpha_0+p-\alpha_1}}\\
&\qquad\mu^{-\alpha_0-1+\frac{\alpha_1-p+1}{p}\frac{(\alpha_0+1)(q-p)}{\alpha_0+p-\alpha_1}+\frac{\alpha_0+1}p\frac{p(\alpha_0+1)-q(\alpha_1-p+1)}{\alpha_0+p-\alpha_1}}J(u)\\
&=J(u).
\end{align*}

Let $(v_n)$ be a maximizing sequence in $X^{1,p}_\infty$ for $B_{p,\alpha_0}$, i.e. $B_{p,\alpha_0}=\lim J(v_n)$. We can suppose $v_n\geq0$. Set $u_n:=(v_n)_{\lambda_n,\mu_n}$, where
\begin{equation*}
\lambda_n=\dfrac{\|v_n\|_{L^p_{\alpha_0}}^{\frac{\alpha_1-p+1}{\alpha_0+p-\alpha_1}}}{\|v'_n\|_{L^p_{\alpha_1}}^{\frac{\alpha_0+1}{\alpha_0+p-\alpha_1}}}\mbox{ and }\mu_n=\dfrac{\|v_n\|_{L^p_{\alpha_0}}^{\frac{p}{\alpha_0+p-\alpha_1}}}{\|v'_n\|_{L^p_{\alpha_1}}^{\frac{p}{\alpha_0+p-\alpha_1}}}.
\end{equation*}
Note that $u_n\in X^{1,p}_\infty$, $\|u_n\|_{L^p_{\alpha_0}}=\|u_n'\|_{L^p_{\alpha_1}}=1$ and $J(u_n)=J(v_n)\overset{n\to\infty}\longrightarrow B_{p,\alpha_0}$. The case $q=p$ is trivial so we can suppose $q>p$. Since $(u_n)$ is bounded in $X^{1,p}_\infty$, (up to subsequence) there exists $u_0\in X^{1,p}_\infty$ such that $u_n\rightharpoonup u_0$ in $X^{1,p}_\infty$. By Proposition \ref{propimersaoinfinito} the following embedding is compact
\begin{equation*}
X^{1,p}_\infty\hookrightarrow L^q_{\alpha_0}(0,\infty).
\end{equation*}
Thus, $u_n\to u_0$ in $L^q_{\alpha_0}$. The weak convergence $u_n\rightharpoonup u_0$ in $X^{1,p}_\infty$ guarantees $\|u_0\|_{L^p_{\alpha_0}},\|u_0'\|_{L^p_{\alpha_1}}\leq 1$. Using
\begin{equation}\label{eq38}
B_{p,\alpha_0}=\lim J(u_n)=\lim\|u_n\|^q_{L^q_{\alpha_0}}=\|u_0\|_{L^q_{\alpha_0}}^q,
\end{equation}
we have $u_0\not\equiv0$. By \eqref{eq38} we conclude
\begin{equation*}
B_{p,\alpha_0}=\|u_0\|^q_{L^q_{\alpha_0}}\leq J(u_0)\leq B_{p,\alpha_0}.
\end{equation*}
Therefore $J(u_0)=B_{p,\alpha_0}$
\end{proof}

\begin{lemma}\label{lemma33}
Suppose $p=2$, $q=4$ and $\alpha_1=1$ in Proposition \ref{lemmab2}. For any $\alpha_0\in(-1,\infty)$ we have
\begin{equation}\label{estimate}
B_{2,\alpha_0}\geq\dfrac{4}{(\alpha_0+1)\mathrm e\log2}.
\end{equation}
In particular, $B_{2,\alpha_0}>2/(\alpha_0+1)$.
\end{lemma}
\begin{proof}
Let $u(r)=\mathrm{e}^{-r^\gamma}$, where $\gamma>0$ will be chosen later. Then
\begin{equation*}
\|u\|_{L^4_{\alpha_0}}^4=\int_0^\infty \mathrm{e}^{-4r^\gamma}r^{\alpha_0}\mathrm dr=\int_0^\infty \mathrm{e}^{-s}\left(\dfrac{s}4\right)^{\frac{\alpha_0+1-\gamma}{\gamma}}\dfrac{\mathrm ds}{4\gamma}=\dfrac{1}{4^{\frac{\alpha_0+1}\gamma}\gamma}\Gamma\left(\dfrac{\alpha_0+1}{\gamma}\right),
\end{equation*}
\begin{equation*}
\|u\|^2_{L^2_{\alpha_0}}=\int_0^\infty \mathrm{e}^{-2r^\gamma}r^{\alpha_0}\mathrm dr=\int_0^\infty \mathrm{e}^{-s}\left(\dfrac{s}2\right)^{\frac{\alpha_0+1-\gamma}{\gamma}}\dfrac{\mathrm ds}{2\gamma}=\dfrac{1}{2^{\frac{\alpha_0+1}\gamma}\gamma}\Gamma\left(\dfrac{\alpha_0+1}{\gamma}\right)
\end{equation*}
and
\begin{equation*}
\|u'\|_{L^2_1}^2=\gamma^2\int_0^\infty \mathrm{e}^{-2r^\gamma}r^{\gamma}r^{\gamma-1}\mathrm dr=\gamma^2\int_0^\infty \mathrm{e}^{-s}\dfrac{s}{2}\dfrac{\mathrm ds}{2\gamma}=\dfrac{\gamma}4\Gamma(2)=\dfrac\gamma4.
\end{equation*}
Thus, by definition of $B_{2,\alpha_0}$,
\begin{equation*}
B_{2,\alpha_0}\geq\dfrac{\|u\|^4_{L^4_{\alpha_0}}}{\|u'\|^2_{L^2_1}\|u\|^2_{L^2_{\alpha_0}}}=\dfrac{4.2^{\frac{\alpha_0+1}{\gamma}}}{\gamma.4^{\frac{\alpha_0+1}{\gamma}}}=\dfrac{2^{2-\frac{\alpha_0+1}\gamma}}{\gamma}.
\end{equation*}
Define $f_{\alpha_0}\colon(0,\infty)\to\mathbb R$ given by $f_{\alpha_0}(\gamma)=2^{2-\frac{\alpha_0+1}{\gamma}}/\gamma$. Since
\begin{equation*}
\gamma^2f'_{\alpha_0}(\gamma)=2^{2-\frac{\alpha_0+1}\gamma}\log2\dfrac{\alpha_0+1}{\gamma^2}\gamma-2^{2-\frac{\alpha_0+1}{\gamma}}=2^{2-\frac{\alpha_0+1}\gamma}\left(\dfrac{\alpha_0+1}\gamma\log2-1\right),
\end{equation*}
$\gamma_0:=(\alpha_0+1)\log2$ is the global maximum of $f_{\alpha_0}$. Therefore,
\begin{equation*}
B_{2,\alpha_0}\geq f_{\alpha_0}(\gamma_0)=\dfrac{2^{2-\frac{1}{\log2}}}{(\alpha_0+1)\log2}=\dfrac{4}{(\alpha_0+1)\mathrm e\log2},
\end{equation*}
which is the desired conclusion.
\end{proof}

\begin{lemma}\label{ilemma23}
It holds that
\begin{equation}\label{i21}
\nu_*\geq\dfrac{\mu^{p-1}}{\Gamma(p)}\eta_*\  (*=0,\infty)
\end{equation}
and
\begin{equation*}
1=\rho_0+\rho_\infty,\ 1\geq\eta_0+\eta_\infty,\ d_{p,\mu}=\nu_0+\nu_\infty.
\end{equation*}
\end{lemma}

\begin{proof}
It is easy to see that Lemma \ref{il21} yields
\begin{align*}
v_*&=\lim_{R\to\infty}\lim_{n\to\infty}\sum_{j=0}^\infty\dfrac{\mu^{p-1+j}}{\Gamma(p+j)}\|u^*_{n,R}\|_{L_{\alpha_0}^{p'(p-1+j)}}^{p'(p-1+j)}\\
&\geq\lim_{R\to\infty}\lim_{n\to\infty}\dfrac{\mu^{p-1}}{\Gamma(p)}\|u^*_{n,R}\|^p_{L^p_{\alpha_0}}=\dfrac{\mu^{p-1}}{\Gamma(p)}\eta_*
\end{align*}
thus \eqref{i21} holds. All others relations follows breaking the integral over $(0,\infty)$ in $(0,R)$ and $(R,\infty)$ and taking $\lim_{R\to\infty}\lim_{n\to\infty}$ along suitable subsequences.
\end{proof}

\begin{lemma}\label{il24}
Let $\rho_*<1$ ($*=0,\infty$). Then we have
\begin{align}
d_{p,\mu}&\|u_{n,R}^*\|^p_{X^{1,p}_\infty}\geq\int_0^\infty\exp_p\left(\mu|u_{n,R}^*|^{p'}\right)r^{\alpha_0}\mathrm dr\nonumber\\
&+\left(\dfrac{1}{\|u^*_{n,R}\|^{p'}_{X^{1,p}_\infty}}-1\right)\int_0^\infty\left(\exp_p\left(\mu|u^*_{n,R}|^{p'}\right)-\dfrac{\mu^{p-1}}{\Gamma(p)}|u^*_{n,R}|^{p}\right)r^{\alpha_0}\mathrm dr.\label{i22}
\end{align}
\end{lemma}
\begin{proof}
By the definition of $d_{p,\mu}$, it is clear that
\begin{align*}
d_{p,\mu}&\geq\int_0^\infty\exp_p\left(\mu\left(\dfrac{u_{n,R}^*}{\|u_{n,R}^*\|_{X^{1,p}_\infty}}\right)^{p'}\right)r^{\alpha_0}\mathrm dr\\
&=\sum_{j=0}^\infty\dfrac{\mu^{p-1+j}}{\Gamma(p+j)}\dfrac{\|u^*_{n,R}\|^{p'(p-1+j)}_{L^{p'(p-1+j)}_{\alpha_0}}}{\|u^*_{n,R}\|^{p'(p-1+j)}_{X^{1,p}_\infty}}\\
&=\dfrac{1}{\|u^*_{n,R}\|^{p}_{X^{1,p}_\infty}}\sum_{j=0}^\infty\dfrac{\mu^{p-1+j}}{\Gamma(p+j)}\dfrac{\|u^*_{n,R}\|^{p'(p-1+j)}_{L^{p'(p-1+j)}_{\alpha_0}}}{\|u^*_{n,R}\|^{pj}_{X^{1,p}_\infty}}\\
&=\dfrac{1}{\|u^*_{n,R}\|^{p}_{X^{1,p}_\infty}}\Bigg[\sum_{j=0}^\infty\dfrac{\mu^{p-1+j}}{\Gamma(p+j)}\|u^*_{n,R}\|^{p'(p-1+j)}_{L_{\alpha_0}^{p'(p-1+j)}}\\
&\quad+\sum_{j=0}^\infty\left(\dfrac{1}{\|u^*_{n,R}\|^{p'j}_{X^{1,p}_\infty}}-1\right)\dfrac{\mu^{p-1+j}}{\Gamma(p+j)}\|u_{n,R}^*\|^{p'(p-1+j)}_{L_{\alpha_0}^{p'(p-1+j)}}\Bigg].
\end{align*}
By multiplying this relation with $\|u_{n,R}^*\|^p_{X^{1,p}_\infty}$, we obtain
\begin{align}
d_{p,\mu}\|u_{n,R}^*\|^p_{X^{1,p}_\infty}&\geq\int_0^\infty\exp_p\left(\mu|u_{n,R}^*|^{p'}\right)r^{\alpha_0}\mathrm dr\nonumber\\
&\quad+\sum_{j=1}^\infty\left(\dfrac{1}{\|u^*_{n,R}\|^{p'j}_{X^{1,p}_\infty}}-1\right)\dfrac{\mu^{p-1+j}}{\Gamma(p+j)}\|u_{n,R}^*\|^{p'(p-1+j)}_{L_{\alpha_0}^{p'(p-1+j)}}.\label{i23}
\end{align}
In view of Lemma \ref{il21} and the assumption $\rho_*<1$, for $j\geq1$, we have
\begin{equation*}
\|u_{n,R}\|_{X^{1,p}_\infty}^{p'j}\leq \|u_{n,R}\|_{X^{1,p}_\infty}^{p'}<1
\end{equation*}
for $R$ and $n$ sufficient large. Hence, for such $R$ and $n$, we see that
\begin{align*}
\sum_{j=1}^\infty&\left(\dfrac{1}{\|u^*_{n,R}\|^{p'j}_{X^{1,p}_\infty}}-1\right)\dfrac{\mu^{p-1+j}}{\Gamma(p+j)}\|u_{n,R}^*\|^{p'(p-1+j)}_{L_{\alpha_0}^{p'(p-1+j)}}\\
&\geq\left(\dfrac{1}{\|u^*_{n,R}\|^{p'}_{X^{1,p}_\infty}}-1\right)\sum_{j=1}^\infty\dfrac{\mu^{p-1+j}}{\Gamma(p+j)}\|u_{n,R}^*\|^{p'(p-1+j)}_{L_{\alpha_0}^{p'(p-1+j)}}\\
&=\left(\dfrac{1}{\|u^*_{n,R}\|^{p'}_{X^{1,p}_\infty}}-1\right)\int_0^\infty\left[\exp_p\left(\mu|u_{n,R}^*|^{p'}\right)-\dfrac{\mu^{p-1}}{\Gamma(p)}|u_{n,R}^*|^{p}\right]r^{\alpha_0}\mathrm dr.
\end{align*}
We have the conclusion by combining this relation with \eqref{i23}.
\end{proof}

\begin{prop}\label{ip23}
Under the assumptions of Proposition \ref{ip22} we have
\begin{equation*}
(\rho_0,\nu_0)=(1,d_{p,\mu})\mbox{ and }(\rho_\infty,\nu_\infty)=(0,0).
\end{equation*}
\end{prop}
\begin{proof}
We divided the proof into two claims.

\medskip

\noindent\textbf{Claim 1:} $\rho_0=0$ or $\rho_0=1$.

\noindent Suppose Claim 1 was false. Then  $0<\rho_0<1$ and, by Lemma \ref{ilemma23}, we also have $0<\rho_\infty<1$. Taking the limits as $n\to\infty$ and $R\to\infty$ in \eqref{i22} (along suitable subsequences) and using Lemma \ref{il21}, we obtain
\begin{equation}\label{i28}
d_{p,\mu}\rho_*\geq\nu_*+\left(\dfrac1{\rho_*^{\frac{1}{p-1}}}-1\right)\left(\nu_*-\dfrac{\mu^{p-1}}{\Gamma(p)}\eta_*\right).
\end{equation}
Lemma \ref{ilemma23} and $0<\rho_*<1$ ($*=0,\infty$) together with \eqref{i28} guarantee
\begin{equation*}
d_{p,\mu}\rho_*\geq\nu_*,\quad*=0,\infty.
\end{equation*}
Summing these inequalities with $*=0$ and $*=\infty$ we get
\begin{equation*}
d_{p,\mu}=d_{p,\mu}(\rho_0+\rho_\infty)\geq\nu_0+\nu_\infty=d_{p,\mu},
\end{equation*}
where we used Lemma \ref{ilemma23} again. Then $d_{p,\mu}\rho_*=\nu_*$, $*=0,\infty$. By \eqref{i28},
\begin{equation*}
\nu_*=\dfrac{\mu^{p-1}}{\Gamma(p)}\eta_*,\quad*=0,\infty.
\end{equation*}
Note that
\begin{equation*}
d_{p,\mu}=\nu_0+\nu_\infty=\dfrac{\mu^{p-1}}{\Gamma(p)}(\eta_0+\eta_\infty)\leq\dfrac{\mu^{p-1}}{\Gamma(p)}
\end{equation*}
constradicts Proposition \ref{ip22}. This concludes the proof of Claim 1.

\medskip

\noindent\textbf{Claim 2:} $\rho_*=0\Rightarrow\nu_*=0$ for $*=0,\infty$.

\noindent Indeed, from Lemma \ref{il21},
\begin{equation*}
\dfrac{1}{\|u_{n,R}^*\|^{p'}_{X^{1,p}_\infty}}-1>\dfrac12
\end{equation*}
for $R$ and $n$ sufficient large. This relation with \eqref{i28} yield
\begin{equation*}
0=d_{p,\mu}\rho_*\geq\nu_*+\dfrac12\left(\nu_*-\dfrac{\mu^{p-1}}{\Gamma(p)}\eta_*\right)\geq \nu_*.
\end{equation*}
Therefore $\nu_*=0$ concludes Claim 2.

\medskip

Finally, suppose the proposition is false. By Claims 1 and 2 we have
\begin{equation*}
(\rho_0,\nu_0)=(0,0)\mbox{ and }(\rho_\infty,\nu_\infty)=(1,d_{p,\mu}).
\end{equation*}
Note that $(u_n)$ is a normalized vanishing sequence. Then, Proposition \ref{ip21} guarantees
\begin{equation*}
d_{p,\mu}=\nu_\infty=\lim_{R\to\infty}\lim_{n\to\infty}\int_R^\infty\exp_p\left(\mu u_n^{p'}\right)r^{\alpha_0}\mathrm dr\leq d_{NVL}(p,\mu)=\dfrac{\mu^{p-1}}{\Gamma(p)},
\end{equation*}
which is absurd in view of Proposition \ref{ip22}.
\end{proof}

\begin{lemma}\label{lemmacompexp}
Suppose $\alpha_0>-1$ and $0<\mu<\alpha_0+1$. Then
\begin{equation*}
\int_0^\infty\left[\exp_p(\mu|u_n|^{p'})-\dfrac{\mu^{p-1}|u_n|^p}{\Gamma(p)}\right]r^{\alpha_0}\mathrm dr\overset{n\to\infty}\longrightarrow\int_0^\infty\left[\exp_p(\mu|u|^{p'})-\dfrac{\mu^{p-1}|u|^p}{\Gamma(p)}\right]r^{\alpha_0}\mathrm dr.
\end{equation*}
\end{lemma}

\begin{proof}
Fix $R>0$ to be chosen later. Note that
\begin{align*}
\Bigg|&\int_0^\infty\left[\exp_p(\mu|u_n|^{p'})-\dfrac{\mu^{p-1}|u_n|^p}{\Gamma(p)}\right]r^{\alpha_0}\mathrm dr-\int_0^\infty\left[\exp_p(\mu|u|^{p'})-\dfrac{\mu^{p-1}|u|^p}{\Gamma(p)}\right]r^{\alpha_0}\mathrm dr\Bigg|\\
&\quad\leq I_1+I_2+I_3+I_4,
\end{align*}
where
\begin{equation*}
    I_1=\int_0^R\left|\mathrm{e}^{\alpha|u_n|^{p'}}-\mathrm{e}^{\alpha|u|^{p'}}\right|r^{\alpha_0}\mathrm dr,\ I_2=\dfrac{\mu^p}{\Gamma(p+1)}\left|\|u_n\|_{L^{pp'}_{\alpha_0}}^{pp'}-\|u\|_{L^{pp'}_{\alpha_0}}^{pp'}\right|,
    \end{equation*}
    \begin{equation*}
    I_3=\sum_{j=2}^\infty\dfrac{\mu^{p-1+j}}{\Gamma(p+j)}\int_R^\infty|u_n|^{(p-1+j)p'}r^{\alpha_0}\mathrm dr\mbox{ and }I_4=\int_R^\infty\exp_p(\mu|u|^{p'})r^{\alpha_0}\mathrm dr.
\end{equation*}
Our proof is completed by showing that $I_1,I_2,I_3,I_4\overset{n\to\infty}\longrightarrow 0$ for $R$ sufficient large.

\medskip

Using the compact embedding $X^{1,p}_\infty\hookrightarrow L^{pp'}_{\alpha_0}$ we have $I_2\to0$.

\medskip

From Lemma \ref{lemma75} we have $\exp_p(\mu|u|^{p'})\in L^1_\theta(R,\infty)$ which implies that $I_4$ is small for $R>0$ sufficient large.

\medskip

It is enough to prove that
\begin{equation}\label{eqsono}
\lim_{n\to\infty}I_1=0,\quad\forall R>0
\end{equation}
and
\begin{equation}\label{eqsono2}
\lim_{R\to\infty}I_3=0\mbox{ uniformly in }n.
\end{equation}

Using Mean Value Theorem we get $v_n(r)$ between $u_n(r)$ and $u(r)$ satisfying $\|v_n\|_{L^{q}_{\alpha_0}}\leq\|u_n\|_{L^q_{\alpha_0}}+\|u_n\|_{L^q_{\alpha_0}}\leq C_q$, for all $q\geq1$, and
\begin{equation*}
    I_1=p'\mu\int_0^Rv_n^{\frac1{p-1}}\mathrm{e}^{\mu v_n^{p'}}|u_n-u|r^{\alpha_0}\mathrm dr.
\end{equation*}
For $p_1,p_3>1$ sufficient large such that $\frac1{p_1}+\frac1{p_2}+\frac1{p_3}=1$ and $p_2\mu<\alpha_0+1$, we have
\begin{align*}
I_1&\leq p'\mu \|v_n\|_{L^\frac{p_1}{p-1}_{\alpha_0}}^{\frac1{p-1}}\left(\int_0^R\mathrm{e}^{p_2\mu v_n^{p'}}r^{\alpha_0}\mathrm dr\right)^{\frac1{p_2}}\|u_n-u\|_{L^{p_3}_{\alpha_0}}\\
&\leq C\|u_n-u\|_{L^{p_3}_{\alpha_0}}\overset{n\to\infty}\longrightarrow0,
\end{align*}
where in the last estimate we used item (b) of Theorem \ref{theo01}. This concludes \eqref{eqsono}.

\medskip

Finally, \eqref{eqsono2} follows by Lemma \ref{lemma53} in the following way
\begin{align*}
    I_3&\leq\sum_{j=2}^\infty\dfrac{(\mu C^{p'})^{p-1+j}}{\Gamma(p+j)}\int_R^\infty r^{\alpha_0-\frac{(\alpha_0+1)(p-1+j)}{p}}\mathrm dr\\
    &\leq\dfrac{pR^{\alpha_0+1}}{\alpha_0+1}\exp_{p+2}(\mu C^{p'}R^{-\frac{\alpha_0+1}{p}})\overset{R\to\infty}\longrightarrow0.
\end{align*}
\end{proof}

\begin{prop}\label{propi25}
Assume the hypotheses of Proposition \ref{ip22}. If $\mu<\alpha_0+1$, then
\begin{equation*}
d_{p,\mu}=\int_0^\infty\exp_p(\mu|u|^{p'})r^{\alpha_0}\mathrm dr.
\end{equation*}
\end{prop}
\begin{proof}
First, we claim that
\begin{equation}\label{eqnhsal}
\lim_{n\to\infty}\int_0^\infty|u_n|^pr^{\alpha_0}\mathrm dr=\int_0^\infty|u|^pr^{\alpha_0}\mathrm dr.
\end{equation}
Indeed, note that, for any $R>0$,
\begin{equation*}
\left|\int_0^\infty(|u_n|^p-|u|^p)r^{\alpha_0}\mathrm dr\right|\leq\left|\int_0^R(|u_n|^p-|u|^p)r^{\alpha_0}\mathrm dr\right|+\int_R^\infty|u_n|^pr^{\alpha_0}\mathrm dr+\int_R^\infty|u|^pr^{\alpha_0}\mathrm dr.
\end{equation*}
Applying Theorem \ref{theo32} and Proposition \ref{ip23} we conclude \eqref{eqnhsal}.

\medskip

Since $(u_n)$ is a maximizing sequence we have
\begin{equation}\label{eqdsl}
    \int_0^\infty\exp_p(\mu|u_n|^{p'})r^{\alpha_0}\mathrm dr\overset{n\to\infty}\longrightarrow d_{p,\mu}.
\end{equation}
Lemma \ref{lemmacompexp} and \eqref{eqnhsal} implies
\begin{align*}
\int_0^\infty\exp_p(\mu|u_n|^{p'})r^{\alpha_0}\mathrm dr&=\int_0^\infty\left[\exp_p(\mu|u_n|^{p'})-\dfrac{\mu^{p-1}|u_n|^p}{\Gamma(p)}\right]r^{\alpha_0}\mathrm dr+\dfrac{\mu^{p-1}}{\Gamma(p)}\int_0^\infty|u_n|^pr^{\alpha_0}\mathrm dr\\
&\overset{n\to\infty}\longrightarrow\int_0^\infty\exp_p(\mu|u|^{p'})r^{\alpha_0}\mathrm dr,
\end{align*}
which together with \eqref{eqdsl} concludes the proof of our proposition.
\end{proof}

\begin{theo}
Let $N\geq2$ and $\alpha_0>-1$. Also, let
\begin{equation*}
B_{2,\alpha_0}=\sup_{u\in X^{1,2}_\infty\backslash\{0\}}\frac{\|u\|^4_{L^4_{\alpha_0}}}{\|u'\|_{L^2_1}^2\|u\|_{L^2_{\alpha_0}}^2}.
\end{equation*}
Then $d_{p,\mu}$ is attained for $0<\mu<\alpha_0+1$ if $N>2$ and for $2/B_{2,\alpha_0}<\mu<\alpha_0+1$ if $N=2$.
\end{theo}
\begin{proof}
By Proposition \ref{propi25}, there exists $u\in X^{1,p}_\infty(\alpha_0,p-1)$ such that
\begin{equation*}
d_{p,\mu}=\int_0^\infty\exp_p\left(\mu|u|^{p}\right)r^{\alpha_0}\mathrm dr.
\end{equation*}
We are left to check that $\|u\|_{X^{1,p}_\infty}=1$. Since $\|u\|_{X^{1,p}_\infty}\leq\liminf\|u_n\|_{X^{1,p}_\infty}=1$, we have
\begin{align*}
d_{p,\mu}&\geq\int_0^\infty\exp_p\left[\alpha\left(\dfrac{u}{\|u\|_{X^{1,p}_\infty}}\right)^{p'}\right]r^{\alpha_0}\mathrm dr=\sum_{j=0}^\infty\dfrac{\mu^{p-1+j}}{\Gamma(p+j)}\dfrac{\|u\|^{p'(p-1+j)}_{L^{p'(p-1+j)}_{\alpha_0}}}{\|u\|^{p'(p-1+j)}_{X^{1,p}_\infty}}\\
&\geq\dfrac1{\|u\|_{X^{1,p}_\infty}^p}\sum_{j=0}^\infty\dfrac{\mu^{p-1+j}}{\Gamma(p+j)}\|u\|_{L^{p'(p-1+j)}_{\alpha_0}}^{p'(p-1+j)}=\dfrac{d_{p,\mu}}{\|u\|^p_{X^{1,p}_\infty}}.
\end{align*}
Therefore $\|u\|_{X^{1,p}_\infty}=1$.
\end{proof}

\section{Proof of Theorem \ref{theo15}}\label{sec4}

The next theorem is a relatively straightforward generalization of \cite[Theorem 1.4]{MR3209335}.

\begin{theo}\label{theoa}
Let $X^{1,p}_{\infty}(\alpha_0,p-1)$ with $\alpha_0>-1$ and $p>1$. Then for each $\mu\in(0,\alpha_0+1)$ there exists $C>0$ depending only on $\mu,$ $p$ and $\alpha_0$ such that
\begin{equation*}
\int_0^\infty\exp_p\left(\mu\left|\dfrac{u(r)}{\|u'\|_{L^p_{p-1}}}\right|^{p'}\right)r^{\alpha_0}\mathrm dr\leq C\left(\dfrac{\|u\|_{L^p_{\alpha_0}}}{\|u'\|_{L^p_{p-1}}}\right)^p
\end{equation*}
for all $u\in X^{1,p}_\infty\backslash\{0\}$.
\end{theo}

\begin{proof}
Let $u\in X^{1,p}_\infty(\alpha_0,p-1)$ with $\|u'\|_{L^p_{p-1}}=1$. Set $r=\mathrm{e}^{-\frac{t}{\alpha_0+1}}$ and $w\colon(-\infty,\infty)\to\mathbb R$ given by $w(t)=(\alpha_0+1)^{\frac{p-1}p}u(r)$. By Lemma \ref{lemma53}, $\lim_{r\to\infty}u(r)=\lim_{t\to-\infty}w(t)=0$. Note that
\begin{equation*}
\int_0^\infty|u(r)|^pr^{\alpha_0}\mathrm dr=\dfrac1{(\alpha_0+1)^p}\int_{-\infty}^\infty|w(t)|^p\mathrm{e}^{-t}\mathrm dt,
\end{equation*}
\begin{equation*}
\int_0^\infty|u'(r)|^pr^{p-1}\mathrm dr=\int_{-\infty}^\infty|w'(t)|^p\mathrm dt
\end{equation*}
and
\begin{equation*}
\int_0^\infty\exp_p\left(\mu|u(r)|^{p'}\right)r^{\alpha_0}\mathrm dr=\dfrac1{\alpha_0+1}\int_{-\infty}^\infty\exp_p\left(\dfrac{\mu}{\alpha_0+1}|w(t)|^{p'}\right)\mathrm{e}^{-t}\mathrm dt.
\end{equation*}
It is sufficient to show that for each $\beta\in(0,1)$ there exists $C_\beta>0$ such that
\begin{equation}\label{ew2}
\int_{-\infty}^\infty\exp_p\left(\beta|w(t)|^{p'}\right)\mathrm{e}^{-t}\mathrm dt\leq C_\beta\int_{-\infty}^\infty|w(t)|^p\mathrm{e}^{-t}\mathrm dt,
\end{equation}
for all $C^1$ functions $w(t)$ satisfying $\lim_{t\to-\infty}w(t)=0$ and $\int_{-\infty}^\infty|w'(t)|^p\mathrm dt=1$.

Consider
\begin{equation*}
A=\{t\in\mathbb R\colon |w(t)|\leq 1\}\mbox{ and }B=\{t\in\mathbb R\colon|w(t)|>1\}.
\end{equation*}
This gives $A\cup B=\mathbb R$ and $A\cap B=\varnothing$. Since $\exp_p(t)\leq C_pt^{p-1}$ for all $t\in[0,1]$, for some $C_p>0$, we have
\begin{equation}\label{ew222}
  \int_A\exp_p\left(\beta|w(t)|^{p'}\right)\mathrm{e}^{-t}\mathrm dt\leq C_p\int_A|w(t)|^p\mathrm{e}^{-t}\mathrm dt.
\end{equation}
By $\lim_{t\to-\infty}w(t)=0$, the set $B$ can be write as $B=\cup_{i\in I}(a_i,b_i)$ with $w(a_i)=1$ $\forall i\in I$, where $I=\mathbb N$ or $I=\{1,\ldots,n\}$ with $b_n=\infty$. Then, for each $i\in I$,
\begin{equation*}
|w(t)|\leq |w(a_i)|+\int_{a_i}^t|w'(s)|\mathrm ds\leq 1+(t-a_i)^{\frac{p-1}p}\quad\forall t\in(a_i,b_i).
\end{equation*}
We recall that
\begin{equation}\label{provisorio}
(a+b)^q\leq (1+\varepsilon)a^q+C_q\left(1+\dfrac1{\varepsilon^{q-1}}\right)b^q,\quad\forall a,b\geq0 \mbox{ and }\varepsilon>0,
\end{equation}
where $C_q>0$ depends only on $q\geq1$. Fixed $\varepsilon>0$ with $\beta(1+\varepsilon)<1$, \eqref{provisorio} guarantees the existence of $C_{\varepsilon,p}>0$ such that
\begin{equation*}
|w(t)|^{p'}\leq(1+\varepsilon)(t-a_i)+C_{\varepsilon,p},\quad\forall t\in (a_i,b_i).
\end{equation*}
Using that
\begin{equation*}
\mathrm{e}^{-a_i}\left(1-\mathrm{e}^{[\beta(1+\varepsilon)-1](b_i-a_i)}\right)=\left(\mathrm{e}^{-a_i}-\mathrm{e}^{-b_i}\right)f(b_i-a_i),
\end{equation*}
where $f(t)=\frac{\mathrm{e}^t-\mathrm{e}^{\beta(1+\varepsilon)t}}{\mathrm{e}^t-1}$ for $t\in\mathbb R$ and $f(\infty)=1$, we have, by Lemma \ref{lemmaphiN},
\begin{align*}
\int_{a_i}^{b_i}\exp_p\left(\beta|w(t)|^{p'}\right)\mathrm{e}^{-t}\mathrm dt&\leq\int_{a_i}^{b_i}\mathrm{e}^{\beta|w(t)|^{p'}-t}\mathrm dt\\
&\leq\int_{a_i}^{b_i}\exp\left[(\beta(1+\varepsilon)-1)(t-a_i)+\beta C_{\varepsilon,p}-a_i\right]\mathrm dt\\
&=\dfrac{\mathrm{e}^{\beta C_{\varepsilon,p}}}{1-\beta(1+\varepsilon)}\left(\mathrm{e}^{-a_i}-\mathrm{e}^{-b_i}\right)f(b_i-a_i).
\end{align*}
Since $f(t)\leq1$ for all $t\in\mathbb R$, we get
\begin{equation*}
\int_{a_i}^{b_i}\exp_p\left(\beta|w(t)|^{p'}\right)\mathrm{e}^{-t}\mathrm dt\leq \dfrac{\mathrm{e}^{\beta C_{\varepsilon,p}}}{1-\beta(1+\varepsilon)}\left(\mathrm{e}^{-a_i}-\mathrm{e}^{-b_i}\right).
\end{equation*}
By $|w|>1$ on $B$, we have $\int_{a_i}^{b_i}|w(t)|^p\mathrm{e}^{-t}\mathrm dt\geq\int_{a_i}^{b_i}\mathrm{e}^{-t}\mathrm dt=\mathrm{e}^{-a_i}-\mathrm{e}^{-b_i}$ and
\begin{equation*}
\sum_{i\in I}\int_{a_i}^{b_i}\exp_p\left(\beta|w(t)|^{p'}\right)\mathrm{e}^{-t}\mathrm dt\leq\dfrac{\mathrm{e}^{\beta C_{\varepsilon,p}}}{1-\beta(1+\varepsilon)}\sum_{i\in I}\int_{a_i}^{b_i}|w(t)|^p\mathrm{e}^{-t}\mathrm dt.
\end{equation*}
Therefore,
\begin{equation}\label{ew22}
  \int_B\exp_p\left(\beta|w(t)|^{p'}\right)\mathrm{e}^{-t}\mathrm dt\leq\dfrac{\mathrm{e}^{\beta C_{\varepsilon,p}}}{1-\beta(1+\varepsilon)}\int_B|w(t)|^p\mathrm{e}^{-t}\mathrm dt,
\end{equation}
where we used the monotone convergence Theorem when $I=\mathbb N$. Taking $C_\beta=\max\{C_p,\mathrm{e}^{\beta C_{\varepsilon,p}}/[1-\beta(1+\varepsilon)]\}$ we have, by \eqref{ew222} and \eqref{ew22},
\begin{equation*}
\int_{-\infty}^\infty\exp_p\left(\beta|w(t)|^{p'}\right)\mathrm{e}^{-t}\mathrm dt\leq C_\beta\int_{-\infty}^\infty|w(t)|^p\mathrm{e}^{-t}\mathrm dt.
\end{equation*}
This gives \eqref{ew2}, which completes the proof of the theorem.
\end{proof}
\begin{cor}\label{corgn}
Let $X^{1,p}_\infty(\alpha_0,p-1)$ with $\alpha_0>-1$ and $p>1$. Then for each $\mu\in(0,\alpha_0+1)$ there exists $C>0$ depending only on $\mu$, $p$ and $\alpha_0$ such that for all $u\in X^{1,p}_\infty$ and $q\geq p$ we have
\begin{equation*}
\|u\|^q_{L^q_{\alpha_0}}\leq \dfrac{C\Gamma(\frac{q}{p'}+1)}{\mu^{\frac{q}{p'}}}\|u'\|^{q-p}_{L^p_{p-1}}\|u\|^p_{L^p_{\alpha_0}}.
\end{equation*}

\end{cor}
\begin{proof}
Set $\widetilde p=\frac{q}{p'}+1$ for a given $q\geq p$. By Lemma \ref{lemmaphiN},
\begin{equation*}
\exp_p(t)\geq\exp_{\widetilde p}(t)\geq \dfrac{t^{\widetilde p-1}}{\Gamma(\widetilde p)}, \quad\forall t\in[0,\infty).
\end{equation*}
Applying this in Theorem \ref{theoa} we conclude the proof.
\end{proof}
Now we are ready to provide the proof of Theorem \ref{theo15} which we based in \cite{MR2854113}.
\begin{proof}[Proof of Theorem \ref{theo15}]
Let $\mathcal S:=\{u\in X^{1,p}_\infty\colon\|u\|_{X^{1,p}_\infty}=1\}$. Fixed $v\in \mathcal S$, we shall check that $v$ is not a critical point for $\mu>0$ small. We consider
\begin{equation*}
v_t(r):=t^{\frac1p}v(t^{\frac{1}{\alpha_0+1}}r)
\end{equation*}
where $t>0$ is a parameter. Set $w_t:=v_t/\|v_t\|_{X^{1,p}_\infty}$ and note that $w_t$ is a curve in $\mathcal S$ with $w_1=v$. Therefore, it is sufficient to show that
\begin{equation}
I:=\dfrac{\mathrm d}{\mathrm dt}\left(\int_0^\infty\exp_p\left(\mu|w_t|^{p'}\right)r^{\alpha_0}\mathrm dr\right)\Bigg|_{t=1}\neq0
\end{equation}
for $\mu>0$ small.

By $\|v_t\|_{L^q_{\alpha_0}}^q=t^{\frac{q-p}p}\|v\|_{L^q_{\alpha_0}}^q$ and $\|v_t'\|_{L^p_{p-1}}^p=t\|v'\|_{L^p_{p-1}}^p$, we have
\begin{align*}
\int_0^\infty\exp_p\left(\mu|w_t|^{p'}\right)r^{\alpha_0}\mathrm dr&=\sum_{j=0}^\infty\dfrac{\mu^{p-1+j}}{\Gamma(p+j)}\dfrac{\|v_t\|^{p'(p-1+j)}_{L^{p'(p-1+j)}_{\alpha_0}}}{\|v_t\|_{X^{1,p}_\infty}^{p'(p-1+j)}}\\
&=\sum_{j=0}^\infty\dfrac{\mu^{p-1+j}}{\Gamma(p+j)}\dfrac{t^{\frac{j}{p-1}}\|v\|^{p'(p-1+j)}_{L^{p'(p-1+j)}_{\alpha_0}}}{\left(\|v\|_{L^p_{\alpha_0}}^p+t\|v'\|^p_{L^p_{p-1}}\right)^{\frac{p-1+j}{p-1}}}.
\end{align*}
Thus,
\begin{align}
I&=\sum_{j=0}^\infty\dfrac{\mu^{p-1+j}}{\Gamma(p+j)}\|v\|^{p'(p-1+j)}_{L^{p'(p-1+j)}_{\alpha_0}}\left(\dfrac{j}{p-1}-\dfrac{p-1+j}{p-1}\|v'\|^p_{L^p_{p-1}}\right)\nonumber\\
&\leq-\dfrac{\mu^{p-1}}{\Gamma(p)}\|v\|^p_{L^p_{\alpha_0}}\|v'\|^p_{L^p_{p-1}}+\sum_{j=1}^\infty\dfrac{j\mu^{p-1+j}}{(p-1)\Gamma(p+j)}\|v\|^{p'(p-1+j)}_{L^{p'(p-1+j)}_{\alpha_0}}\nonumber\\
&=\dfrac{\mu^{p-1}}{\Gamma(p)}\|v\|^p_{L^p_{\alpha_0}}\|v'\|^p_{L^p_{p-1}}\left(-1+\sum_{j=1}^\infty\dfrac{j\mu^{j}\Gamma(p-1)}{\Gamma(p+j)}\dfrac{\|v\|^{p'(p-1+j)}_{L^{p'(p-1+j)}_{\alpha_0}}}{\|v\|^p_{L^p_{\alpha_0}}\|v'\|^p_{L^p_{p-1}}}\right).\label{eqw1}
\end{align}
By Corollary \ref{corgn}, there exists $C_{\alpha_0,p}>0$ depending only on $\alpha_0$ and $p$ such that
\begin{equation*}
\|v\|^{p'(p-1+j)}_{L^{p'(p-1+j)}_{\alpha_0}}\leq \dfrac{C_{\alpha_0,p}\Gamma(p+j)}{(\frac{\alpha_0+1}{2})^{p-1+j}}\|v'\|^{p'j}_{L^p_{p-1}}\|v\|^p_{L^p_{\alpha_0}}.
\end{equation*}
We conclude from $p\leq2$ and $\|v\|_{X^{1,p}_\infty}=1$ that
\begin{equation*}
\dfrac{\|v\|^{p'(p-1+j)}_{L^{p'(p-1+j)}_{\alpha_0}}}{\|v\|^p_{L^p_{\alpha_0}}\|v'\|^p_{L^p_{p-1}}}\leq \dfrac{C_{\alpha_0,p}\Gamma(p+j)}{(\frac{\alpha_0+1}{2})^{p-1+j}}\|v'\|^{p'j-N}_{L^p_{p-1}}\leq\dfrac{C_{\alpha_0,p}\Gamma(p+j)}{(\frac{\alpha_0+1}{2})^{p-1+j}},\quad\forall j\geq1.
\end{equation*}
This and \eqref{eqw1} gives
\begin{equation*}
I\leq\dfrac{\mu^{p-1}}{\Gamma(p)}\|v\|^p_{L^p_{\alpha_0}}\|v'\|^p_{L^p_{p-1}}\left[-1+\dfrac{\mu\Gamma(p-1)C_{\alpha_0,p}}{(\frac{\alpha_0+1}2)^p}\sum_{j=1}^\infty j\left(\dfrac{2\mu}{\alpha_0+1}\right)^{j-1}\right].
\end{equation*}
Taking $\mu<\min\{\frac{\alpha_0+1}3,1/C\}$, where $C=\frac{\Gamma(p-1)C_{\alpha_0,p}}{[(\alpha_0+1)/2]^p}\sum_{j=1}^\infty j(\frac{2}3)^{j-1}$, we have
\begin{equation*}
I<\dfrac{\mu^{p-1}}{\Gamma(p)}\|v\|^p_{L^p_{\alpha_0}}\|v'\|^p_{L^p_{p-1}}\left(-1+\mu C\right)<0,
\end{equation*}
which completes the proof.
\end{proof}

\end{document}